\documentclass[oneside,11pt,reqno]{amsart}
\usepackage{amsmath,amssymb,amsthm,xypic} 
\input xy
\xyoption{all}

\usepackage{hyperref}
\usepackage[shortcuts]{extdash}

\usepackage[totalheight=23 true cm, totalwidth=16 true cm]{geometry}
\usepackage{color}

\title{Expansive dynamics on profinite groups}
\author{Michael Wibmer}
\address{Michael Wibmer, Institute of Analysis and Number Therory, Graz University of Technology, Kopernikusgasse~24, 8010 Graz, Austria, \url{https://sites.google.com/view/wibmer}}
\email{wibmer@math.tugraz.at}

\thanks{This work was supported by the NSF grants DMS-1760212, DMS-1760413, DMS-1760448 and the Lise Meitner grant M 2582-N32 of the Austrian Science Fund FWF}


\date{\today}

\newtheorem{theo}{Theorem}[section]
\newtheorem{lemma}[theo]{Lemma}
\newtheorem{prop}[theo]{Proposition}
\newtheorem{cor}[theo]{Corollary}
\newtheorem{defi}[theo]{Definition}
\newtheorem{rem}[theo]{Remark}

\newtheorem*{theononumber}{Theorem A}

\theoremstyle{definition}
\newtheorem{ex}[theo]{Example}

\newcommand{\f}{\phi}

\newcommand{\Z}{\mathbb{Z}}

\newcommand{\N}{\mathbb{N}}
\newcommand{\G}{\mathcal{G}}

\newcommand{\id}{\operatorname{id}}

\newcommand{\s}{\sigma}

\newcommand{\ld}{\operatorname{ld}}

\newcommand{\nn}{\mathbb{N}}

\def\H{\mathcal{H}}
\renewcommand{\sc}{{\sigma o}}
\newcommand{\Gm}{\mathbb{G}_m}

\begin{document}

\maketitle


\begin{abstract}
A profinite group equipped with an expansive endomorphism is equivalent to a one-sided group shift. We show that these groups have a very restricted structure. More precisely, we show that any such group can be decomposed into a finite sequence of full one-sided group shifts and two finite groups. 
\end{abstract}

\section{Introduction} 

An endomorphism $\s\colon G\to G$ of a profinite group $G$ is \emph{expansive} if there exists an open subgroup $U$ of $G$ with  
$\bigcap_{n\in \N}\s^{-n}(U)=1$. There are two obvious examples: A (discrete) finite group with an arbitrary endomorphism (choose $U=1$) and a full one-sided group shift on a finite group $\G$, i.e., $G=\G^\N$ with $\s(g_0,g_1,\ldots)=(g_1,g_2,\ldots)$ (choose $U=1\times\G\times \G\times\ldots$). Our main result shows that any profinite group with an expansive endomorphism is build up from these two examples. More precisely, we have (Theorem~\ref{theo: main}):

\begin{theononumber} Let $G$ be a profinite group equipped with an expansive endomorphism $\s$. Then there exists a subnormal series
$$G\supseteq G_1\supseteq G_2\supseteq\ldots\supseteq G_n$$
of closed $\s$-stable subgroups $G_i$ of $G$ such that
\begin{itemize}
	\item $G_i/G_{i+1}$ is isomorphic to a full one-sided group shift on a finite simple group $\G_i$ for $i=1,\ldots,n-1$,
	\item  $G/G_1$ is a finite group and $\s\colon G/G_1\to G/G_1$ is an automorphism,
	\item $G_n$ is a finite group and some power of $\s\colon G_n\to G_n$ is the trivial endomorphism $g\mapsto 1$.
\end{itemize} 		
Moreover, the length $n$ of such a series, the group $G_1$ and the isomorphism classes of the finite simple groups $\G_i$ are uniquely determined by $G$. 
%
%
%
\end{theononumber}

A continuous map $\s\colon X\to X$ on a metric space $(X,d)$ is \emph{expansive} if there exists an $\varepsilon >0$ such that for any two distinct points $x,y\in X$ there is an $n\in\N$ with $d(\s^n(x),\s^n(y))>\varepsilon$. 
In the context of a continuous group homomorphism $\s\colon G\to G$ on a topological group $G$, this condition translates to the existence of a neighborhood $U$ of $1$ such that for any two distinct $g,h\in G$ there is an $n\in\N$ with $\s^n(g)\notin \s^n(h)U$, or equivalently, $h^{-1}g\notin\s^{-n}(U)$. Thus, $\s$ is expansive if and only if there is a neighborhood $U$ of $1$ with $\bigcap_{n\in \N}\s^{-n}(U)=1$. Similarly, an \emph{automorphism} $\s$ of a topological group is an \emph{expansive automorphism} if there exists a neighborhood $U$ of $1$ such that $\bigcap_{n\in \Z}\s^{-n}(U)=1$. 
 
 Expansive automorphisms of topological groups have been studied extensively under varying restrictions on the group (e.g., profinite, compact or locally compact). See
 \cite{Kitchens:ExpansiveDynamicsOnZeroDimensionalGroups}, \cite{Fagnani:SomeResultsOnTheClassificationOfExpansiveAutomorphismsOfCompactAbelianGroups},
  \cite{KitchensSchmidt:AutomorphismsOfCompactGroups},  \cite{BoyleSchrauder:ZdgroupshiftsAndBernoulliFactors}, \cite{GloecknerRaja:ExpansiveAutomorphismsOfTotallyDisconnectedLocallyCompactGroups}, \cite{Shah:ExpansiveAutomorphismsOnLocallyCompactGroups}, \cite[Chapter 3]{Schmidt:DynamicalSystemsOfAlgebraicOrigin} and the references given there.
 
A fundamental result, concerning expansive automorphisms of profinite groups, due to B. Kitchens (see \cite{Kitchens:ExpansiveDynamicsOnZeroDimensionalGroups} or \cite[Section 6.3]{Kitchens:SymbolicDynamics}), is that any such group is topologically conjugate to a direct product of a finite set equipped with a bijection and a full two-sided shift. See \cite{BoyleSchrauder:ZdgroupshiftsAndBernoulliFactors} for a discussion of higher dimensional analogs and \cite{Sobottka:TopologicalQuasiGroupShifts} for a generalization to quasi-groups.

In this article, we also establish a one-sided analog of Kitchens' result (Theorem \ref{theo: topological conjugacy}): If $G$ is a profinite group equipped with an expansive endomorphism $\s\colon G\to G$, then there exists an integer $n\geq 0$ such that $\s^n(G)$ is topologically conjugate to a finite set equipped with a bijection and a full one-sided shift.

Despite its beautiful simplicity, Kitchens' result is not fully satisfactory since the topological conjugacy in general does not respect the group structure. In fact, it appears that the problem, set forth by Kitchens in \cite{Kitchens:ExpansiveDynamicsOnZeroDimensionalGroups}, to classify all expansive automorphisms of profinite groups up to isomorphism is still wide open. 

We also establish a version of Theorem A for expansive automorphisms (Theorem \ref{theo: expansive auto}), i.e., a two-sided version. The statement is very similar. The only significant difference is that the group $G_n$ does not occur in the two-sided version. One can think of our two-sided version of Theorem~A as a variant of Kitchens' result that respects the group structure. It reduces the problem of classifying all expansive automorphisms of profinite groups to the study of the group extension problem for profinite groups with an expansive automorphism. This is somewhat similar to how the Jordan-H\"{o}lder theorem reduces the study of finite groups to the study of finite simple groups and the group extension problem for finite groups. Our proof of the uniqueness part of Theorem~A is actually modeled on the proof of the Jordan-H\"{o}lder theorem.   

There is no direct connection between profinite groups equipped with an expansive endomorphism and profinite groups equipped with an expansive automorphism. Indeed, if $G$ is a profinite group and $\s\colon G\to G$ is a map that is simultaneously an expansive endomorphism and an expansive automorphism, then $G$ must be finite (Corollary \ref{cor: s injective implies finite}). However, there is a universal construction $G\rightsquigarrow G^*$ that associates to any profinite group $G$ equipped with an expansive endomorphism, a profinite group $G^*$ equipped with an expansive automorphism. We use this universal construction to deduce the two-sided version of Theorem A from Theorem A.

We also present an application of Theorem A to difference algebra. Babbitt's decomposition theorem (\cite[Theorem 5.4.13]{Levin}) is an important classical theorem in difference algebra that elucidates the structure of Galois extensions of difference fields. Here a difference field is a field equipped with an endomorphism. The connection to expansive endomorphisms is that the Galois group of an extension of difference fields is naturally a profinite group equipped with an endomorphism. If the extension of difference fields is finitely generated, then the endomorphism on the Galois group is expansive. Indeed, one can think of Theorem A as a group version of Babbitt's decomposition theorem. Based on Theorem A, we present a new proof of Babbitt's decomposition theorem that yields additional information concerning the uniqueness of the decomposition (Theorem \ref{theo: Babbitt new}).   

A certain connection between difference algebra and symbolic dynamics, to be detailed in a forthcoming paper, was discovered by Ivan Toma\v{s}i\'{c}. While we do not use or even explain this connection here, this paper would not have happened without it and the author is grateful to Ivan Toma\v{s}i\'{c} for sharing his discovery. In the light of this connection, the results of this article could also be interpreted as results about a certain class of affine difference algebraic groups.

\medskip

We conclude this introduction with an overview of the article: In Section 2 we discuss one-sided group shifts and show that they are always of finite type. In Section 3 we first explain the equivalence of categories between the category of one-sided group shifts and the category of profinite groups equipped with an expansive endomorphism. We then study the latter category in more detail. In particular, we discuss quotients and analogs of the isomorphism theorems and the Schreier refinement theorem. We also introduce the $\s$-identity component that has properties somewhat similar to the identity component of an algebraic group. Section 4 contains the technical heart of the paper and establishes Theorem A. In Section 5 the one-sided analog of Kitchens' result is proved. Then the two-sided version of Theorem A is established in Section~6. Finally, Section 7 contains the application to Babbitt's decomposition theorem.

\section{One-sided group shifts}
\label{sec: onesided group shift}

In this section we provide some basic results concerning one-sided group shifts. In particular, we show that every one-sided group shift is of finite type. We begin by fixing the notation and recalling some basic definitions from symbolic dynamics. See \cite{Kitchens:SymbolicDynamics} or \cite{LindMarcus:IntroductionToSymbolicDynamisAndCoding}. As general conventions, the set $\N$ of natural numbers contains zero and a subset $Y$ of a set $X$ equipped with a map $\s\colon X\to X$ is \emph{$\s$-stable} if $\s(Y)\subseteq Y$.

Let $\mathcal{A}$ be a finite set. We consider $\mathcal{A}^\N$ as a topological space via the product topology of the discrete topology on $\mathcal{A}$. The topological space  $\mathcal{A}^\N$ together with the continuous map $\s\colon \mathcal{A}^\N\to  \mathcal{A}^\N$, given by $\s(a_0,a_1,\ldots)=(a_1,a_2,\ldots)$ is the \emph{full one-sided shift} on the alphabet $\mathcal{A}$. A \emph{one-sided shift space} or \emph{one-sided shift} on $\mathcal{A}$ is a closed subset $X$ of $\mathcal{A}^\N$ such that $\s(X)\subseteq X$. A morphism between two one-sided shifts $X$ and $Y$ (possibly on different alphabets) is a continuous map $\f\colon X\to Y$ such that  
$$
\xymatrix{
	X \ar_\s[d] \ar^\f[r] & Y \ar^\s[d]  \\
	X   \ar^\f[r] & Y	
}
$$
commutes. A \emph{word} or \emph{block} of length $i$ is a sequence of $i$ elements from $\mathcal{A}$. A one-sided shift $X$ on $\mathcal{A}$ is a (one-sided) \emph{subshift of finite type} if there exists a finite set $\mathcal{F}$ of blocks such that $X$ consists of all elements of $\mathcal{A}^\N$ that do not contain any blocks from $\mathcal{F}$.

An important class of subshifts of finite type is formed by those coming from directed graphs, also called $1$-step subshifts of finite type. Let $\Gamma$ be a directed graph with set of vertices equal to $\mathcal{A}$. Then the set $X_\Gamma\subseteq\mathcal{A}^\N$ of all sequences in $\mathcal{A}$ that trace out the vertices of an infinite directed path in $\Gamma$, is a subshift of finite type.

The \emph{topological entropy} of a one-sided shift $X$ on $\mathcal{A}$ is
$$h(X)=\lim_{i\to\infty}\frac{\log(d_i)}{i},$$
where $d_i$ denotes the cardinality of the image of $X\to \mathcal{A}^i,\ (a_0,a_1,\ldots)\mapsto (a_0,a_1,\ldots,a_{i-1})$, i.e., the cardinality of all blocks of length $i$ that occur in elements of $X$.

In this article, we are mainly interested in the situation when the alphabet is a finite group. If $\G$ is a finite group, then $\G^\N$ inherits a group structure via componentwise multiplication. Indeed  $\G^\N$ is a profinite group and $\s\colon \G^\N\to\G^\N$ is a continuous group homomorphism. In this situation, the pair $(\G^\N,\s)$ is called the \emph{full one-sided group shift} on $\G$.
A \emph{one-sided group shift} $G$ on $\G$ is a one-sided shift on $\G$ such that $G$ is a subgroup of $\G^\N$. A morphism between two one-sided group shifts is a morphism between one-sided shifts that is a group homomorphism. We note that (in the two-sided context) group shifts are called Markov subgroups in \cite[Section 6.3]{Kitchens:SymbolicDynamics}. Here we follow the nomenclature from \cite{LindMarcus:IntroductionToSymbolicDynamisAndCoding} and \cite{BoyleSchrauder:ZdgroupshiftsAndBernoulliFactors}.   

Let $G$ be a one-sided group shift on the finite group $\G$. The following notation will be useful: For $i\in\N$, let $G[i]$ denote the image of the group homomorphism 
$G\to \G^{i+1},\ (g_0,g_1,\ldots)\mapsto (g_0,g_1,\ldots,g_i)$. Then $G[i]$ is the subgroup of $\G^{i+1}$ consisting of all blocks of length $i+1$ that occur in elements of $G$.
For $i\geq 1$, the map $$\pi_i\colon G[i]\to G[i-1],\ (g_0,\ldots,g_i)\mapsto (g_0,\ldots,g_{i-1})$$ is a surjective group homomorphism.

Every two-sided group shift is a subshift of finite type (\cite[Section 6.3]{Kitchens:SymbolicDynamics}). As we now show, a similar result holds for one-sided group shifts. The proof has some important consequences.

\begin{prop} \label{prop: group shift is of finite type}
	Every one-sided group shift is a subshift of finite type.
\end{prop}
\begin{proof}
	Let $\G$ be a finite group and $G\leq \G^\N$ a one-sided group shift on $\G$. Set $\G_0=G[0]$ and for $i\geq 1$, let $\G_i\subseteq \G$ denote the follower set of $(1,\ldots,1)\in\G^i$, i.e., $\ker(\pi_i)=\{(1,\ldots,1,g)\in\G^{i+1} |\ g\in \G_i \}$. So $\G_i$ is a subgroup of $\G$. As $G$ is stable under the shift map $\s\colon \G^\N\to\G^\N$, we also have group homomorphisms $\s_i\colon  G[i]\to G[i-1],\ (g_0,\ldots,g_i)\mapsto (g_1,\ldots,g_i)$. Since $\s_i$ maps $\ker(\pi_i)$ injectively into $\ker(\pi_{i-1})$, we see that $\G_{i-1}\subseteq \G_i$.
	We thus have a descending chain $\G_0\supseteq \G_1\supseteq \G_2\supseteq \ldots$ of subgroups of $\G$ that must eventually stabilize. Let $\G'=\bigcap_{i\in\N}\G_i$ denote the eventual value and let $n\in\N$ be minimal with the property that $\G_i=\G'$ for all $i\geq n$.
	Set $\mathcal{H}=G[n]$ and let $G'\subseteq \G^\N $ denote the subshift of finite type that avoids all blocks from $\mathcal{F}=G[n]\smallsetminus \mathcal{H}$.
	In other words, $G'$ is the subgroup of $\G^\N$ consisting of all elements that have all their blocks of length $n+1$ inside $\mathcal{H}$. By construction $G\subseteq G'$.
	
	We claim that $G=G'$. We have $G[i]\leq G'[i]$ for $i\in\N$. To show that $G=G'$ it suffices to show that $G[i]=G'[i]$ for $i\in\N$. This is clear for $i=0,\ldots,n$ and we will prove the general case by induction on $i$. So we assume that $i\geq n$ and that $G[i]=G'[i]$. We have to show that $G[i+1]=G'[i+1]$. There is a commutative diagram
	$$
	\xymatrix{
		G[i+1] \ar_{\pi_{i+1}}[d]  \ar@{^(->}[r] & G'[i+1]  \ar^{\pi'_{i+1}}[d] \\
		G[i] \ar@{=}[r]& G'[i]
	}
	$$
	where the vertical maps $\pi_{i+1}$ and $\pi'_{i+1}$ are the surjective group homomorphisms given by projection onto the first $i+1$ coordinates. It suffices to show that $\ker(\pi_{i+1})=\ker(\pi'_{i+1})$. Clearly $\ker(\pi_{i+1})\subseteq\ker(\pi'_{i+1})$. Moreover, $\ker(\pi_{i+1})=\{1\}^{i+1}\times\G'\leq\G^{i+2}$ since $i+1\geq n$. Let $h=(1,\ldots,1,g)\in\ker(\pi'_{i+1})\leq\G^{i+2}$. By definition of $G'$, the element $(1,\ldots,1,g)\in\G^{n+1}$ lies in $\mathcal{H}=G[n]$ and so it lies in the kernel of $\pi_n\colon\mathcal{H}=G[n]\to G[n-1]$. We conclude that $g\in\G'$ and $h\in \ker(\pi_{i+1})$ as desired.
\end{proof}

%

\begin{cor} \label{cor: entropy and limit degree}
	Let $G$ be a one-sided group shift. Then $h(G)=\log(d)$ for some integer $d\geq 1$.
\end{cor}
\begin{proof}
	We use the notation of the proof of Proposition \ref{prop: group shift is of finite type} and furthermore set $d=|\G'|$. For $i\geq 1$ we have $|G[i]|=|G[i-1]|\cdot|\G_{i}|$ and so inductively, $$|G[i]|=|\G_0|\cdot|\G_1|\ldots|\G_{i}|=|\G_0|\ldots|\G_{n-1}|\cdot|\G_n|^{i-n+1}=|\G_0|\ldots|\G_{n-1}|\cdot d^{i-n+1}$$
	for $i\geq n$.
	 Thus 
	 $$h(G)=\lim_{i\to\infty}\frac{\log|G[i]|}{i+1}=\lim_{i\to\infty}\frac{\log(|\G_0|\ldots|\G_{n-1}|\cdot d^{-n})}{i+1}+\lim_{i\to\infty}\frac{\log(d^{i+1})}{i+1}=\log(d).$$
\end{proof}

%

\begin{defi}
	Let $G\leq \G^\N$ be a one-sided group shift on the finite group $\G$ and $n\geq 1$. Then $G$ is an \emph{$n$-step group shift of finite type} if there exists a subgroup $\mathcal{H}$ of $\G^{n+1}$ such that $G$ consists of exactly those elements of $\G^\N$ that have all blocks of length $n+1$ inside $\mathcal{H}$. 
\end{defi}
The following corollary is immediate from the proof of Proposition \ref{prop: group shift is of finite type}.
\begin{cor} \label{cor: every group shift is Nstep}
Every one-sided group shift is an $n$-step group shift of finite type for some $n\ge 1$.\qed
\end{cor}	

$1$-step subshifts of finite type are described by directed graphs. The following definition provides a group version of this well-known fact. 

\begin{defi} \label{def: group graph}
	Let $\G$ be a finite group. A \emph{directed group graph} on $\G$ is a directed graph $\Gamma$ such that the set of vertices of $\Gamma$ equals $\G$ and such that the set of directed edges of $\Gamma$ is a subgroup of $\G\times \G$.
\end{defi}
For a directed group graph $\Gamma$ on $\G$ let $$G_\Gamma=X_\Gamma\subseteq \G^\N$$ denote the subshift of finite type defined by $\Gamma$. Then $G_\Gamma$ is a $1$-step group shift of finite type. Conversely, every $1$-step group shift $G\subseteq \G^\N$ is of the form $G=G_\Gamma$ for some directed group graph $\Gamma$.

In a directed group graph the set of directed edges is a group. The same is true for the set of directed paths of a fixed length (finite or infinite): Two directed paths are multiplied by multiplying the vertices and the directed edges. We will see later (Lemma \ref{lemma: exists standard embedding}) that every one-sided group shift is isomorphic to some $G_\Gamma$.
%

\section{Expansive profinite groups}

If $G$ is a one-sided group shift, then $G$ is a profinite group and $\s\colon G\to G$ is an expansive endomorphism. Conversely, we will see that every profinite group with an expansive endomorphism is isomorphic to a one-sided group shift. It is sometimes beneficial to work inside this larger category of expansive profinite groups. For example, if $G$ and $N$ are one-sided group shifts on a finite group $\G$ such that $N$ is a normal subgroup of $G$, then the quotient $G/N$ is not a one-sided group shift on the nose as there is no canonical choice of the alphabet for $G/N$. On the other hand, it is clear that $G/N$ is a profinite group equipped with an endomorphism (which can be shown to be expansive).

In this section we provide some basic definitions and results concerning expansive profinite groups that will then be used in the next section to establish the main decomposition theorem (Theorem \ref{theo: main}). In particular, we define the $\s$-identity component and the limit degree of an expansive profinite group.

\subsection{Expansive profinite groups versus one-sided group shifts}
Let $G$ be a profinite group. A continuous group homomorphism  $\s\colon G\to G$ is called an \emph{expansive endomorphism} if there exists a neighborhood $U$ of $1\in G$ such that $\bigcap_{n\in \N}\s^{-n}(U)=1$. Since the open normal subgroups of $G$ are a neighborhood basis for $1$ (\cite[Theorem~2.1.3]{RibesZaleskii:ProfiniteGroups}), we can assume that $U$ is an open normal subgroup of $G$.

\begin{defi} \label{defi: expansive profinite group}
 An \emph{expansive profinite group} is a profinite group $G$ together with an expansive endomorphism $\s\colon G\to G$.
\end{defi}
We will usually omit $\s$ from the notation and simply refer to $G$ as an expansive profinite group.
A morphism between expansive profinite groups $G$ and $H$ is a continuous group homomorphism $\f\colon G\to H$ such that
$$
\xymatrix{
G \ar_\s[d] \ar^\f[r] & H \ar^\s[d]  \\
G   \ar^\f[r] & H	
}
$$
commutes.

\begin{ex} \label{ex: group  shift is expansive}
	Let $G\leq\G^\N$ be a one-sided group shift on the finite group $\G$. Then $G$ is an expansive profinite group. Indeed, we can choose $U=\{(g_0,g_1,\ldots)\in G|\ g_0=1\}$.
\end{ex}

See Lemma \ref{lemma: expansive structure on Galois group} for an explanation how expansive profinite groups naturally occur in the study of Galois extensions of difference fields. The following lemma provides a converse to Example~\ref{ex: group  shift is expansive}.
\begin{lemma} \label{lemma: expansive profinite group is group shift}
	Every expansive profinite group is isomorphic to a one-sided group shift. 
\end{lemma}
\begin{proof}
	Let $G$ be an expansive profinite group and let $U$ be an open normal subgroup of $G$ such that  $\bigcap_{n\in \N}\s^{-n}(U)=1$. Then $\G=G/U$ is a finite group and 
	$$\f\colon G\to\G^\N,\ g\mapsto (g,\s(g),\s^2(g),\ldots)$$
	is a group homomorphism that commutes with $\s$. It is injective because  $\bigcap_{n\in \N}\s^{-n}(U)=1$.
	
	 The inverse image of a basic open subset $$V=V(h_0,\ldots,h_n)=\{(g_0,g_1,\ldots)\in\G^\N|\ g_0=h_0,\ldots,g_n=h_n\}$$ of $\G^\N$, where $h_1,\ldots,h_n\in \G$ are $U$-cosets in $G$, equals $h_0\cap \s^{-1}(h_1)\cap\ldots\cap \s^{-n}(h_n)$, which is open $G$. Thus $\f$ is continuous. 
	
	A continuous group homomorphism between profinite groups is closed (\cite[Remark~1.2.1~(e)]{FriedJarden:FieldArithmetic}). In particular, $\f(G)\subseteq\G^\N$ is closed. As $\f(G)$ is stable under $\s$, we see that $\f(G)$ is a one-sided group shift. Because $\f$ is closed it follows that $\f\colon G\to \f(G)$ is a homeomorphism and thus an isomorphism. 
\end{proof}

From the above lemma and example we see that the concepts ``expansive profinite group'' and ``one-sided group shift'' are interchangeable. More precisely we have:

\begin{cor}
	The category of one-sided group shifts is equivalent to the category of expansive profinite groups.
\end{cor}
\begin{proof}
	Every one-sided group shift is an expansive profinite group (Example \ref{ex: group  shift is expansive}). Since in both categories the morphisms are defined in the same fashion, it suffices to know that every expansive profinite group is isomorphic to a one-sided group shift. This is exactly Lemma \ref{lemma: expansive profinite group is group shift}.
\end{proof}

We can thus think of expansive profinite groups as a ``coordinate free" version of one-sided group shifts. One advantage of working with the larger category of expansive profinite groups is that certain constructions, such as quotients by normal closed $\s$-stable subgroups, are more naturally performed in this category.

\subsection{Group theory for expansive profinite groups} \label{subsec: Group theory for expansive profinite groups}

The isomorphism theorems for abstract groups carry over without difficulty to the category of expansive profinite groups. The same holds for the Schreier refinement theorem, which will be the key for establishing the uniqueness in our main decomposition theorem (Theorem \ref{theo: main}). 

The following lemma shows that subgroups and quotients of expansive profinite groups are well-behaved.

\begin{lemma} \label{lemma: expansive subgroups and quotients}
	Let $G$ be an expansive profinite group.
	\begin{enumerate}
		\item If $H$ is a closed $\s$-stable subgroup of $G$, then $H$ (with the induced topology and endomorphism) is an expansive profinite group.
		\item If $N$ is a normal closed $\s$-stable subgroup of $G$, then $G/N$ (with the quotient topology and induced endomorphism) is an expansive profinite group and the canonical map $G\to G/N$ is a morphism of expansive profinite groups. 
	\end{enumerate}	
\end{lemma}
\begin{proof}
	A closed subgroup of a profinite group is a profinite group (\cite[Proposition 2.2.1]{RibesZaleskii:ProfiniteGroups}). If $U$ is an open subgroup of $G$ such that $\bigcap_{n\in\N}\s^{-n}(U)=1$, then $U'=H\cap U$ is an open subgroup of $H$ and $\bigcap_{n\in\N}{\s'}^{-n}(U')=1$, where $\s'\colon H\to H$ is the restriction of $\s\colon G\to G$. This proves (i).
	
	The quotient of a profinite group by a closed normal subgroup is a profinite group (\cite[Proposition 2.2.1]{RibesZaleskii:ProfiniteGroups}). Because $\s(N)\subseteq N$ we have a well-defined continuous group homomorphism $\overline{\s}\colon G/N\to G/N,\ gN\mapsto \s(g)N$. 
	
	To show that $\overline{\s}$ is expansive, we may assume that $G$ is a one-sided group shift on a finite group $\G$ (Lemma \ref{lemma: expansive profinite group is group shift}). Then, it follows from Corollary \ref{cor: every group shift is Nstep} that $N\leq \G^\N$ is an $n$-step group shift of finite type for some $n\geq 1$. Note that because $N$ is normal in $G$, $N[i]$ is normal in $G[i]$ for every $i\in \N$. In particular, $N[n]$ is normal in $G[n]$. Set $\H=G[n]/N[n]$ and 
	$\f\colon G\to \H^\N,\ (g_0,g_1,\ldots)\mapsto (\overline{(g_i,g_{i+1},\ldots,g_{i+n})})_{i\in\N}$. Then $\f$ is a morphism of expansive profinite groups with kernel $N$. Set $U=\f^{-1}(1\times \H\times\H\times\ldots)$ and $\overline{U}=U/N\subseteq G/N$. Then $\overline{U}$ is an open subgroup of $G/N$ and if $g\in G$ is such that $\overline{g}\in \bigcap_{i\in \N}\overline{\s}^{-i}(\overline{U})$, i.e., $\s^i(g)\in U$  for all $i\in \N$, then $\s^i(\f(g))\in 1\times \H\times\H\times\ldots$ for all $i$ and thus $\f(g)=(1,1\ldots)$. Therefore $g\in N$ and $\bigcap_{i\in \N}\overline{\s}^{-i}(\overline{U})=1$.	
\end{proof}

\begin{defi}
An \emph{expansive subgroup} of an expansive profinite group is	a closed $\s$\=/stable subgroup.  
\end{defi}
By Lemma \ref{lemma: expansive subgroups and quotients} (i) an expansive subgroup $H$ of an expansive profinite group $G$ is an expansive group in its own right. If $H$ is normal in $G$ we will speak of a normal expansive subgroup.

\begin{prop}[Isomorphism theorems for expansive profinite groups] \label{prop: isomorphosm theorems} \mbox{} 
	\begin{enumerate} 
		\item Let $\f\colon G\to H$ be a morphism of expansive profinite groups. Then $\f(G)$ is an expansive subgroup of $H$, $\ker(\f)$ is a normal expansive subgroup of $G$ and  the canonical map $G/\ker(\f)\to \f(G)$ is an isomorphism of expansive profinite groups.
		\item Let $N$ be a normal expansive subgroup of an expansive group $G$ and $\pi\colon G\to G/N$ the canonical map. Then the map
		$$\{\text{expansive subgroups of $G$ containing $N$}\}\longrightarrow\{\text{expansive subgroups of $G/N$}\},$$ $H\mapsto \pi(H)=H/N$
		is a bijection with inverse $H'\mapsto \pi^{-1}(H')$. Moreover $H$ is normal in $G$ if and only if $H/N$ is normal in $G/N$ and in that case $G/H\simeq (G/N)/(H/N)$.
		\item Let $H$ and $N$ be expansive subgroups of an expansive profinite group $G$ such that $H$ normalizes $N$. Then $HN$ is an expansive subgroup of $G$, $H\cap N$ is a normal expansive subgroup of $H$ and $HN/N\simeq H/H\cap N$.
	\end{enumerate}
	
\end{prop}
\begin{proof}
	The isomorphism theorems hold for profinite groups. See e.g., \cite[Section 1.2]{FriedJarden:FieldArithmetic}. (The key observation here is that any morphism of profinite groups is a closed map.) One immediately verifies that the relevant constructions are compatible with the endomorphism~$\s$.
\end{proof}

\begin{defi}
	A \emph{subnormal series} of an expansive profinite $G$ is a sequence
	\begin{equation} \label{eq: subnormal series}
	G=G_0\supseteq G_1\supseteq G_2\supseteq \ldots\supseteq G_n=1
	\end{equation}	
	of expansive subgroups $G_i$ of $G$ such that $G_{i+1}$ is a normal expansive subgroup of $G_i$ for $i=0,\ldots,n-1$.
	Another subnormal series 
		\begin{equation}\label{eq: subnormal series 2}
		G=H_0\supseteq H_1\supseteq H_2\supseteq \ldots\supseteq H_m=1	\end{equation}
	is a \emph{refinement} of (\ref{eq: subnormal series}) if $\{H_0,\ldots,H_m\}\subseteq \{G_0,\ldots,G_n\}$. 
	
	Two subnormal series (\ref{eq: subnormal series}) and (\ref{eq: subnormal series 2}) are \emph{equivalent} if $m=n$ and there exists a permutation $\pi$ such that $G_i/G_{i+1}$ is isomorphic to $H_{\pi(i)}/H_{\pi(i)+1}$ for $i=0,\ldots,n-1$.	
\end{defi}
We will sometimes omit the first group $G=G_0$ and the last group $G_n=1$ in our notations for a subnormal series for $G$.

Similarly to the isomorphism theorems, the Schreier refinement theorem carries over in a straight forward fashion from the category of abstract groups to the category of expansive profinite groups.
For the sake of completeness we include a sketch of the proof (see e.g., \cite[Theorem 5.11]{Rotman:AnIntroductionToTheTheoryOfGroups}).

\begin{lemma} \label{lemma: Zassenhaus}
	Let $N_1, H_1, N_2, H_2$ be expansive subgroups of an expansive profinite group such that $N_i$ is a normal subgroup of $H_i$ for $i=1,2$. Then $N_1(H_1\cap N_2)$ is a normal expansive subgroup of $N_1(H_1\cap H_2)$, $N_2(N_1\cap H_2)$ is a normal expansive subgroup of $N_2(H_1\cap H_2)$ and 
	$$\frac{N_1(H_1\cap H_2)}{N_1(H_1\cap N_2)}\simeq\frac{N_2(H_1\cap H_2)}{N_2(N_1\cap H_2)}.$$	
\end{lemma}
\begin{proof}
	The statement about normality holds for abstract groups (\cite[Lemma 5.10]{Rotman:AnIntroductionToTheTheoryOfGroups}), so it also holds in our context. As $H_1\cap H_2$ normalizes $N_1(H_1\cap N_2)$ it follows from Proposition \ref{prop: isomorphosm theorems} (iii) that 
	
	\begin{equation} \label{eq: isom}
	\frac{(H_1\cap H_2)N_1(H_1\cap N_2)}{N_1(H_1\cap N_2)}\simeq \frac{H_1\cap H_2}{(H_1\cap H_2)\cap N_1(H_1\cap N_2)}
	\end{equation}
	as expansive profinite groups. But $(H_1\cap H_2)N_1(H_1\cap N_2)=N_1(H_1\cap H_2)$ and $(H_1\cap H_2)\cap N_1(H_1\cap N_2)=(H_1\cap N_2)(N_1\cap H_2)$. Thus (\ref{eq: isom})
	simplifies to
	$$\frac{N_1(H_1\cap H_2)}{N_1(H_1\cap N_2)}\simeq \frac{H_1\cap H_2}{(H_1\cap N_2)(N_1\cap H_2)}.$$
	By symmetry (exchanging the indices $1$ and $2$) also 
	$$\frac{N_2(H_1\cap H_2)}{N_2(N_1\cap H_2)}\simeq \frac{H_1\cap H_2}{(H_1\cap N_2)(N_1\cap H_2)}$$
	and the claim follows.	
\end{proof}

\begin{prop} \label{prop: schreier refinement}
Any two subnormal series of an expansive profinite group have equivalent refinements. 
\end{prop}	
\begin{proof}
Let 
	\begin{equation} \label{eq: subnormal 1}
		G=G_0\supseteq G_1\supseteq \ldots\supseteq  G_n=1
	\end{equation}
and 
\begin{equation} \label{eq: subnormal 2}
G=H_0\supseteq H_1\supseteq \ldots \supseteq H_m=1
\end{equation}
be subnormal series for an expansive profinite group $G$. Setting $G_{i,j}=G_{i+1}(G_i\cap H_j)$ for $i=0,\ldots,n-1$ and $j=0,\ldots,m$ yields a refinement of (\ref{eq: subnormal 1}). Similarly, setting $H_{j,i}=H_{j+1}(G_i\cap H_j)$ for $j=0,\ldots,m-1$ and $i=0,\ldots,n$ yields a refinement of (\ref{eq: subnormal 2}). By Lemma~\ref{lemma: Zassenhaus} we have $G_{i,j}/G_{i,j+1}\simeq H_{j,i}/H_{j,i+1}$ and so the two refinements are equivalent.
\end{proof}

\subsection{The limit degree of an expansive profinite group}

The limit degree of an expansive profinite group is the topological entropy but conveniently transformed so that it always is an integer. It is a rough measure for the ``seize'' of an expansive profinite group. In this section we provide some basic properties of this numerical invariant. The main decomposition theorem (Theorem \ref{theo: main}) will be proved by induction on the limit degree.

\begin{defi}
	Let $G$ be an expansive profinite group. As the topological entropy of subshifts of finite type is invariant under isomorphism, we can define $\ld(G)$, the \emph{limit degree} of $G$ as $\exp(h(G'))$, where $G'$ is any one-sided group shift isomorphic to $G$. 
\end{defi}

Note that by Corollary \ref{cor: entropy and limit degree} the limit degree is a positive integer. Moreover, it has the following interpretation:

\begin{lemma} \label{lemma: ld}
	Let $G$ be a one-sided group shift on a finite group $\G$. Then the sequence $(|\ker(G[i]\xrightarrow{\pi_i} G[i-1])|)_{i\geq 1}$ is non-increasing and stabilizes with value $\ld(G)$. If $n$ is the smallest integer such that $|\ker(G[i]\xrightarrow{\pi_i} G[i-1])|=\ld(G)$, then $G$ is $n$-step.
\end{lemma}
\begin{proof}
	This is clear from the proof of Proposition \ref{prop: group shift is of finite type} and Corollary \ref{cor: entropy and limit degree}.
\end{proof}

\begin{ex} \label{ex: ld of full shift}
	Let $G=\G^\N$ be the full one-sided group shift on a finite simple group, then $\ld(G)=|\G|$.
\end{ex}

Note that the limit degree of a one-sided group shift $G$ can also be describes as $\ld(G)=\lim_{i\to\infty}\frac{|G[i]|}{|G[i-1]|}$. The smallest value $\ld(G)$ can take is $1$. The following lemma shows that this happens if and only if $G$ is finite.

\begin{lemma} \label{lemma: finite if and only if ld 1}
	Let $G$ be an expansive profinite group. Then $\ld(G)=1$ if and only if $G$ is finite. 
\end{lemma}
\begin{proof}
	If $G$ is finite, we can consider the underlying finite group $\G$ obtained from $G$ by forgetting $\s$. Then $\f\colon G\to \G^\N,\ g\mapsto (g,\s(g),\s^2(g),\ldots)$ identifies $G$ with a one-sided group shift of finite type with limit degree $1$. 
	
	Conversely, assume that $G$ is an expansive profinite with $\ld(G)=1$. We may assume that $G$ is a one-sided group shift on a finite group $\G$. Then, using the notation of the proof of Proposition~\ref{prop: group shift is of finite type} and Corollary \ref{cor: entropy and limit degree}, we have $|G[i]|=|\G_0|\cdot|\G_1|\ldots|\G_{i}|=|\G_0|\ldots|\G_{n-1}|$ for $i\geq n$. In particular, $m:=|G[i]|$ is independent of $i$ for $i\gg 0$. Since $G$ is the projective limit of the $G[i]$'s, it follows that $|G|=m$. 
\end{proof}

The following lemma shows that the limit degree is multiplicative.

\begin{lemma} \label{lemma: ld multiplicative on quotients}
	Let $G$ be an expansive profinite group and $N$ a normal expansive subgroup. Then $\ld(N)$ divides $\ld(G)$ and $\ld(G/N)=\frac{\ld(G)}{\ld(N)}$.
\end{lemma}
\begin{proof}
	By Lemma \ref{lemma: expansive profinite group is group shift} we may assume that $G$ is a one-sided group shift on a finite group $\G$. Let $n\in\N$ be such that $N$ is an $n$-step group shift of finite type (Corollary \ref{cor: every group shift is Nstep}).
	The morphism $$\f\colon G\to (G[n]/N[n])^\N,\ (g_0,g_1,\ldots)\mapsto (\overline{(g_i,g_{i+1},\ldots,g_{i+n})})_{i\in\N}$$
	of expansive profinite groups has kernel $N$ and induces, for every $i\in\N$, an isomorphism
	$G[n+i]/N[n+i]\simeq \f(G)[i]$. Therefore
	$$\ld(G/N)=\ld(\f(G))=\lim_{i\to\infty}\frac{|\f(G)[i]|}{|\f(G)[i-1]|}=\lim_{i\to\infty}\frac{\frac{|G[n+i]|}{|N[n+i]|}}{\frac{|G[n+i-1]|}{|N[n+i-1]|}}=\lim_{i\to\infty}\frac{\frac{|G[n+i]|}{|G[n+i-1]|}}{\frac{|N[n+i]|}{|N[n+i-1]|}}=\frac{\ld(G)}{\ld(N)}.$$
\end{proof}


Every expansive profinite group is isomorphic to $G_\Gamma$ for some directed group graph $\Gamma$ (Definition \ref{def: group graph}). For later use, we record a slightly more precise statement.

\begin{lemma} \label{lemma: exists standard embedding}
	Let $G$ be an expansive profinite group. Then there exists a finite group $\H$ and a $1$-step group shift $H\subseteq \H^\N$ such that $G$ is isomorphic to $H$, $H[0]=\H$ and $\ld(G)=|\ker(H[1]\xrightarrow{\pi_1} H[0])|$. In particular, $G\simeq G_\Gamma$, for some directed group graph $\Gamma$.
\end{lemma}
\begin{proof}
	It is well-known (\cite[p. 27]{Kitchens:SymbolicDynamics}) that every one-sided subshift of finite type is isomorphic to a $1$-step subshift of finite type (via a higher block representation). We follow a similar idea here. 
	
	By Lemma \ref{lemma: expansive profinite group is group shift} we can assume that $G$ is a one-sided group shift on a finite group $\G$. 
%
%
Let $n$ be the smallest integer such that $|\ker(G[n]\xrightarrow{\pi_n}G[n-1])|=\ld(G)$. Then $G\leq \G^\N$ consists of exactly those sequences in $\G^\N$ that have all blocks of length $n+1$ inside $\H=G[n]$ (Lemma~\ref{lemma: ld}). Define a map $\f\colon G\to \H^\N$ by $\f(g)=(g_i,g_{i+1},\ldots,g_{i+n})_{i\in\N}$ for $g=(g_0,g_1,\ldots)\in G$. Then $\f$ is an injective morphism of group shifts. The image $H=\f(G)$ of $\f$ is the $1$-step group shift on $\H$ defined by the directed group graph with directed edges $E\leq \H\times \H$ given by $$E=\{((g_0,\ldots,g_n),(g'_0,\ldots,g'_n))\in\H\times \H|\ g_1=g'_0,\ldots,g_n=g'_{n-1}\}.$$
	So $\f\colon G\to H$ is an isomorphism. As the projection $G\to G[n]$ is surjective, we see that $H[0]=\H$. Moreover, the kernel of $\pi_1\colon H[1]=E\to H[0]=\H,\ (h_0,h_1)\mapsto h_0$ equals
	$\{1\}\times \ker(G[n]\xrightarrow{\pi_n}G[n-1])\leq\H\times\H$
	and thus has cardinality $\ld(G)$.
\end{proof}

The following corollary shows that an expansive endomorphism is rarely injective.

\begin{cor} \label{cor: s injective implies finite}
	Let $G$ be an expansive profinite group with $\s\colon G\to G$ injective. Then $G$ is finite.
\end{cor}
\begin{proof}
	By Lemma \ref{lemma: exists standard embedding} we can assume that $G=G_\Gamma$ for some directed group graph $\Gamma$. Without loss of generality, we can assume that every vertex of $\Gamma$ is contained in an infinite directed path. Then $\s$ is injective on $G$ if and only if there is no vertex with more than one incoming directed edge. This implies that $\Gamma$ is a disjoint union of directed cycles. But then $G$ is finite. 
\end{proof}

We will have use for a version of Lemma \ref{lemma: exists standard embedding} that simultaneously works for an expansive subgroup.

\begin{lemma} \label{lemma: simultaneously one-step}
	Let $G$ be an expansive profinite group and $H\leq G$ an expansive subgroup. Then there exists a finite group $\G$ and an injective morphism $\f\colon G\to \G^\N$ of expansive profinite groups such that $\f(G)\leq\G^\N$ and $\f(H)\leq\G^\N$ are $1$-step group shifts of finite type.
\end{lemma}
\begin{proof}
	By Lemma \ref{lemma: expansive profinite group is group shift} we can assume that $G$ is a one-sided group shift on a finite group $\H$. By Corollary \ref{cor: every group shift is Nstep} there exists an $n_G\in\N$ such that $G$ is $n_G$-step. Similarly, there exists an $n_H\in\N$ such that $H$ is $n_H$-step. Let $n$ be the maximum of $n_G$ and $n_H$. Then $G$ and $H$ are both $n$-step.
	
	Set $\G=G[n]$. Then $\f\colon G\to \G^\N,\ (g_0,g_1,\ldots)\mapsto ((g_i,g_{i+1},\ldots,g_{i+n}))_{i\in\N}$ has the desired property (cf. proof of Lemma \ref{lemma: exists standard embedding}). 
\end{proof}

\subsection{The $\sigma$-identity component}

Since profinite groups are totally disconnected, the connected component containing the identity is always trivial. However, requiring the closed sets of an expansive profinite group to also be $\s$-stable, leads to an interesting notion of identity component with properties somewhat analogous to the identity component of an algebraic group. The $\s$-identity component is important for the main decomposition theorem (Theorem \ref{theo: main}) because it yields the first group in the subnormal series.


Some considerations in this section have some similarity with \cite[Section 5.1]{Kitchens:SymbolicDynamics} and \cite[Section 4.4]{LindMarcus:IntroductionToSymbolicDynamisAndCoding}. However, our approach is different and guided by topology.

Let $X$ be a topological space equipped with a continuous map $\s\colon X\to X$. Then the closed $\s$-stable subsets of $X$ satisfy the axioms for the closed sets of a topology. We call this topology on $X$ the \emph{$\s$-topology}. 

Recall that a connected component of a topological space is a maximal connected subset. The connected components are closed and the whole space is the disjoint union of its connected components.
A subset of $X$ that is connected or irreducible with respect to the $\s$-topology is called \emph{$\s$-connected} or \emph{$\s$-irreducible} respectively.
The connected components with respect to the $\s$-topology are called the \emph{$\s$-connected components}.

%

An infinite subshift of finite type has infinitely many connected components as the connected components are in fact the singletons. However, the following lemma shows that it only has finitely many $\s$-connected components. 

\begin{lemma} \label{lemma: sconnected components of subshift of finite type}
	Let $X$ be a one-sided subshift of finite type. Then $X$ has only finitely many $\s$-connected components. 
\end{lemma}
\begin{proof}
	Every one-sided subshift of finite type is isomorphic to a $1$-step subshift of finite type (cf. \cite[Proposition 2.3.9]{LindMarcus:IntroductionToSymbolicDynamisAndCoding}). We can thus assume that $X=X_\Gamma$ for a directed graph $\Gamma$.
	Without loss of generality, we assume that every vertex of $\Gamma$ lies on an infinite directed path.
	
	Let $\Gamma_1,\ldots,\Gamma_r$ denote the strongly connected components of $\Gamma$ whose associated subshift $X_{\Gamma_i}$ is non-empty. Note that the subshift associated to a strongly connected component is empty only if the component has a unique vertex and no directed edge (from the vertex to itself), and this only happens for vertices that do not lie on any directed circuit in $\Gamma$.

	We define an undirected graph $\Delta$ with set of vertices $\{\Gamma_1,\ldots,\Gamma_r\}$ and an edge between $\Gamma_i$ and $\Gamma_j$ if there exists a directed path in $\Gamma$ 
	that connects a vertex in $\Gamma_i$ to a vertex in $\Gamma_j$. Let $\Delta_1,\ldots,\Delta_s$ denote the connected components of $\Delta$. For $i=1,\ldots,s$ let $\Theta_i$ denote the full directed subgraph of $\Gamma$ whose vertices are all vertices of $\Gamma$ that can be connected with a directed path to a vertex belonging to some $\Gamma_j$, such that $\Gamma_j$ is a vertex of $\Delta_i$.

	An infinite directed path in $\Gamma$ might traverse from one strongly connected component to another. However, in that case, it can never return to this strongly connected component. Therefore, every infinite directed path in $\Gamma$ eventually stays in one $\Gamma_i$. This shows that every infinite directed path in $\Gamma$ lives inside a unique $\Theta_i$. Therefore $X=X_{\Theta_1}\uplus\ldots\uplus X_{\Theta_s}$. To complete the proof it suffices to show that $X_{\Theta_i}$ is $\s$-connected for $i=1,\ldots,s$.
		
	The closure of a $\s$-orbit is $\s$-irreducible. It follows that $X_\Gamma$ is $\s$-irreducible if $\Gamma$ is strongly connected because then $X_\Gamma$ contains a point with dense $\s$-orbit. To find such a point one constructs a directed path in $\Gamma$ that traverses every finite directed path in $\Gamma$ (cf. \cite[Theorem 1.4.1 (i)]{Kitchens:SymbolicDynamics}).
	
	It follows that $X_{\Gamma_j}$ is $\s$-irreducible and a fortiori $\s$-connected for every $j=1,\ldots,r$. Note that $X_{\Gamma_j}\subseteq X_{\Theta_i}$ for every $\Gamma_j$ that is a vertex of $\Delta_i$. For every $\Gamma_j$ that is a vertex of $\Delta_i$ there exists a unique $\s$-connected component $X_j$ of $X$ containing $X_{\Gamma_j}$ because $X_{\Gamma_j}$ is $\s$-connected. We will show that $X_j=X_{j'}$ for any $j,j'$ such that $\Gamma_j$ and $\Gamma_{j'}$ are vertices of $\Delta_i$. Because $\Delta_i$ is connected, it suffices to show this under the additional assumption that there exists a directed path in $\Gamma$ from a vertex in $\Gamma_{j'}$ to a vertex in $\Gamma_{j}$.
	
	Note that if a point $x\in X$ is such that $\s^m(x)\in X_{j}$ for some $m\in\N$, then $x\in X_j$ because the $\s$-connected component containing $x$ also contains $\s^m(x)$. Using this, we see that a given point $y\in X_{\Gamma_{j'}}$ can be approximated arbitrarily well with a point in $X_{j}$: Choose a directed path in $X_{\Gamma_{j'}}$ that agrees with the directed path corresponding to $y$ up to an arbitrary large index and then continue this directed path to an infinite directed path in $\Gamma$ that eventually stays in $\Gamma_j$. Because $X_j$ is closed, it follows that $y\in X_j$. So $y\in X_j\cap X_{j'}$ and therefore $X_j=X_{j'}$ as desired.
	
	Thus there exists a $\s$-connected component $Y$ of $X$ such that $X_{\Gamma_j}\subseteq Y$ for every $j$ such that $\Gamma_j$ a vertex of $\Delta_i$. Since every infinite directed path in $\Theta_i$ eventually stays in some $\Gamma_j$, with $\Gamma_j$ a vertex of $\Delta_i$, we see that for every $x\in X_{\Theta_i}$ there exists an $m\in \N$ with $\s^m(x)\in Y$. This shows that  $X_{\Theta_i}\subseteq Y$. From $X=X_{\Theta_1}\uplus\ldots\uplus X_{\Theta_s}$ we deduce that $X_{\Theta_i}=Y$ is $\s$-connected.
\end{proof}
%
We now return to groups:

\begin{lemma} \label{lemma: sconnected component}
	Let $G$ be an expansive profinite group. Then:
	\begin{enumerate}
		\item There are only finitely many $\s$-connected components in $G$.
		\item The $\s$-connected component $G^\sc$ of $G$ that contains $1$ is a normal expansive subgroup of $G$ such that $G/G^\sc$ is finite and $\s\colon G/G^\sc\to G/G^\sc$ is bijective.
		\item If $G=G_\Gamma$ for a directed group graph $\Gamma$, then the $\s$-connected component containing $1$ equals $G_{\Gamma^o}$, where $\Gamma^o$ is the full directed subgraph of $\Gamma$ whose set of vertices consists of all vertices of $\Gamma$ that can be connected to $1$ with a directed path. 
	\end{enumerate}

\end{lemma}
\begin{proof}
	Since every expansive profinite group is isomorphic to a group shift of finite type, (i) follows from Lemma \ref{lemma: sconnected components of subshift of finite type}.
	To establish (ii), by Lemma \ref{lemma: exists standard embedding}, we may assume that $G=G_\Gamma$ for some directed group graph $\Gamma$ on a finite group $\G$. Without loss of generality we assume that every vertex of $\Gamma$ is contained in an infinite directed path.

	Let $\G^o\subseteq \G$ denote the set of all vertices $g$ of $\Gamma$ such that there exists a directed path in $\Gamma$ starting at $g$ and ending at $1$. If $g_1,g_2\in \G^o$, then a directed path from $g_i$ to $1$ can be extended by adding the vertex $1$ at the end a certain number of times. The product of two such directed paths then yields a directed path from $g_1g_2$ to $1$. This shows that $\G^o$ is a subgroup of $\G$.
	
	To show that $\G^o$ is normal in $\G$, let $g\in \G^o$ and fix a directed path $\gamma$ from $g$ to $1$. Let $h\in\G$ and choose a directed path $\delta$ starting at $h$ with the same length as $\gamma$. Then $\delta\gamma\delta^{-1}$ is a directed path in $\Gamma$ from $hgh^{-1}$ to $1$. Thus $hgh^{-1}\in \G^o$ and $\G^o$ is normal in $\G$.
	
	Let $\Gamma^o$ denote the full directed subgraph of $\Gamma$ with vertex set $\G^o$. Because $\G^o$ is a normal subgroup of $\G$, we see that $G_{\Gamma^o}$ is a normal expansive subgroup of $G=G_\Gamma$. If $g\in \G$ is such that there exists a directed edge from $g$ to an element of $\G^o$, then $g\in\G^o$. This shows that $\s^{-1}(G_{\Gamma^o})=G_{\Gamma^o}$ and therefore $\s\colon G/G_{\Gamma^o}\to G/G_{\Gamma^o}$ is injective. It thus follows from Corollary~\ref{cor: s injective implies finite} that $G/G_{\Gamma^o}$ is finite and so $\s\colon G/G_{\Gamma^o}\to G/G_{\Gamma^o}$ is bijective. Thus $G/G_{\Gamma^o}$ is the disjoint union of $\s$-orbits $o_1,\ldots,o_r$. If $\pi\colon G\to G/G^\sc$ is the canonical map, then $G$ is the disjoint union of the closed $\s$-stable subsets $\pi^{-1}(o_1),\ldots,\pi^{-1}(o_r)$. Note that the identity element of $G/G^\sc$ is a $\s$-orbit, say $o_1$. Then $\pi^{-1}(o_1)=G_{\Gamma^o}$. From the proof of Lemma \ref{lemma: sconnected components of subshift of finite type} it is clear that $G_{\Gamma^o}$ is $\s$-connected. From $G=\pi^{-1}(o_1)\uplus\ldots\uplus\pi^{-1}(o_r)$ it follows that $G_{\Gamma^o}$ is the $\s$-connected component of $G$ containing $1$, i.e., $G^\sc=G_{\Gamma^o}$. This completes the proof of (ii) and~(iii).
\end{proof}

\begin{defi}
	Let $G$ be an expansive profinite group. The $\s$-connected component $G^\sc$ of $G$ containing $1$ is called the \emph{$\s$-identity component} of $G$. 
\end{defi}

By Lemma \ref{lemma: sconnected component} we know that $G^\sc$ is a normal expansive subgroup of $G$ such that $G/G^\sc$ is finite and $\s\colon G/G^\sc\to G/G^\sc$ is bijective. Note that an expansive profinite group is $\s$-connected if and only if it equals its $\s$-identity component.

\begin{ex} \label{ex: full shift is sconnected}
	A full one-sided group shift is $\s$-connected by Lemma \ref{lemma: sconnected component} (iii). 
\end{ex}

\begin{ex} \label{ex: main for finite}
	Let $G$ be an expansive profinite group that is finite.
	We claim that $G^\sc=\{g\in G|\ \exists \ n\in \N \colon \s^n(g)=1\}.$ Let $N$ denote the right hand side of this equation. Then $N$ is a normal $\s$-connected expansive subgroup of $G$. Moreover, $\s^{-1}(N)=N$, so $\s\colon G/N\to G/N$ is injective and therefore bijective. It follows that $G$ is the disjoint union of the $\s$-stable sets $\pi^{-1}(o)$, where $o$ is an orbit of $\s$ on $G/N$ and $\pi\colon G\to G/N$ the canonical map. From this we deduce that $N$ is the $\s$-connected component of $G$ containing $1$.
\end{ex}

The following simple example shows that a $\s$-connected expansive profinite group need not be $\s$-irreducible. However, it follows from Corollary \ref{cor: expansive group is full shift} that a $\s$-connected expansive profinite group $G$ with $\s\colon G\to G$ surjective is $\s$-irreducible.

\begin{ex}
	Let $G=\{1,h,h^2\}$ be the cyclic group with three elements and $\s\colon G\to G,\ g\mapsto 1$ the trivial endomorphism. Then $G$ is $\s$-connected but not $\s$-irreducible because $G$ is the union of the $\s$-closed sets $\{1,h\}$ and $\{1,h^2\}$.
\end{ex}

For the proof of the main decomposition theorem we need the following:

\begin{lemma} \label{lemma: Nsc is normal}
	Let $G$ be an expansive profinite group and $N$ a normal expansive subgroup of $G$. Then $N^\sc$ is normal in $G$.
\end{lemma}
\begin{proof}
By Lemma \ref{lemma: simultaneously one-step} we may assume that $G\leq \G^\N$ such that $G$ and $N$ are $1$-step group shifts of finite type. Let $\Gamma$ be the directed group graph on $\G$ such that $G=G_\Gamma$. Then there is a directed subgraph $\Gamma'$ of $\Gamma$ such that $N=G_{\Gamma'}$.
 The set $\mathcal{N}'=N[0]$ of vertices of $\Gamma'$ is a normal subgroup of $G[0]\leq\G$ because $N$ is normal in $G$.
  Moreover, $N^\sc=G_{\Gamma''}$, where $\Gamma''$ is the full directed subgraph of $\Gamma'$ whose set of vertices $\mathcal{N}''$ consists of all elements of $\mathcal{N}'$ that can be connected with $1$ through a directed path inside $\Gamma'$ (Lemma \ref{lemma: sconnected component} (iii)).

We will show that $\mathcal{N}''$ is normal in $G[0]$. Let $n\in\mathcal{N}''$ and $g\in G[0]$. Consider an infinite directed path $\gamma$ in $\Gamma$ starting at $g$ and an infinite directed path $\delta$ in $\Gamma'$ that starts at $n$, goes to $1$ and then stabilizes at $1$. Because $N$ is normal in $G$, the directed path $\gamma\delta\gamma^{-1}$ lies inside $\Gamma'$. Moreover it starts at $gng^{-1}$ and ends at $1$. Therefore $gng^{-1}\in\mathcal{N}''$ and $\mathcal{N''}$ is normal in $G[0]$. This implies that $N^\sc=G_{\Gamma''}$ is normal in $G=G_{\Gamma}$. 
\end{proof}

\begin{lemma} \label{lemma: quotient of sconnected is sconnected}
	Let $\f\colon G\to H$ be a morphism of expansive profinite groups. If $G$ is $\s$-connected, then $\f(G)$ is $\s$-connected.
\end{lemma}
\begin{proof}
	If $V$ is a closed $\s$-stable subset of $H$, then $\f^{-1}(V)$ is a closed $\s$-stable subset of $G$. Thus $\f$ is $\s$-continuous, i.e., continuous with respect to the $\s$-topologies on $G$ and $H$. So the claim follows from the general fact that a continuous map sends connected subsets to connected subsets.
\end{proof}

\begin{lemma} \label{lemma: ld for sc}
	Let $G$ be an expansive profinite group. Then $\ld(G^\sc)=\ld(G)$. 
\end{lemma}
\begin{proof}
	Since $G/G^\sc$ is finite (Lemma \ref{lemma: sconnected component}), we know that $\ld(G/G^\sc)=1$ from Lemma~\ref{lemma: finite if and only if ld 1}. Thus the claim follows from Lemma \ref{lemma: ld multiplicative on quotients}.
\end{proof}

The following two lemmas are needed to prove the uniqueness of the group $G_1$ in the theorem stated in the introduction. The next lemma is a converse to Lemma \ref{lemma: sconnected component} (ii).

\begin{lemma} \label{lemma: sconnected converse}
	Let $N$ be a normal expansive subgroup of an expansive profinite group $G$ such that $G/N$ is finite and $\s\colon G/N\to G/N$ is bijective. Then $G^\sc\subseteq N$. Moreover, if $N$ is $\s$-connected, then $N=G^\sc$.
\end{lemma}
\begin{proof}
	Consider the canonical map $\pi\colon G\to G/N$ and let $o_1,\ldots,o_r$ denote the orbits of $\s$ on $G/N$. Because $G/N$ is finite, these $\s$-orbits are closed and because $\s$ is bijective on $G/N$ the orbits are disjoint. If follows that $G$ is the disjoint union of the closed  $\s$-stable subsets $\pi^{-1}(o_i)$.
	Since $N=\pi^{-1}(o_1)$, where $o_1=\{1\}$ is the $\s$-orbit of the identity of $G/N$, we see that $G^\sc\subseteq N$. If $N$ is $\s$-connected, then $\pi^{-1}(o_1)$ is a $\s$-connected component. So $N=G^\sc$.
\end{proof}

\begin{lemma} \label{lemma: sconnected in short exact sequence}
Let $G$ be an expansive profinite group with a normal expansive subgroup $N$ such that $N$ and $G/N$ are $\s$-connected. Then $G$ is $\s$-connected.
\end{lemma}
\begin{proof}
	As $N$ is $\s$-connected and contains $1$, we have $N\subseteq G^\sc$. Thus $G^\sc/N$ is a normal expansive subgroup of $G/N$ with $(G/N)/(G^\sc/N)\simeq G/G^\sc$. So $(G/N)/(G^\sc/N)$ is finite and $\s\colon (G/N)/(G^\sc/N)\to (G/N)/(G^\sc/N)$ is bijective (Lemma \ref{lemma: sconnected component}).
	Because $G^\sc/N$ is $\s$\=/connected (Lemma \ref{lemma: quotient of sconnected is sconnected}) it follows from Lemma \ref{lemma: sconnected converse} that $G^\sc/N=(G/N)^\sc=G/N$. Therefore $G^\sc=G$.	
%
\end{proof}

%
%


\subsection{One-sided group shifts on finite simple groups}

The results in this subsection are needed for the uniqueness statement in the main decomposition theorem.

%
%
%

\begin{lemma} \label{lemma: normal subgroup of shift on simple group}
	Let $\G$ be a finite simple group and $G=\G^\N$ the full one-sided group shift on $\G$. If $N$ is a proper normal expansive subgroup of $G$, then $N$ is finite and $G/N$ is isomorphic to $\G^\N$.
\end{lemma}
\begin{proof}
	We have to distinguish two cases: First we assume that $\G$ is not abelian. This implies that any normal subgroup $\mathcal{N}$ of $\G^n$ is of the form $\mathcal{N}=\mathcal{N}_1\times\ldots\times\mathcal{N}_n$ with $\mathcal{N}_i\in\{1,\G\}$ for $i=1,\ldots,n$. To see this, note that if $(h_1,\ldots,h_n)\in\mathcal{N}$ with $h_i\neq 1$ for some $1\leq i\leq n$, then there exists $g$ in $\G$ with $gh_i\neq h_ig$ as otherwise the center of $\G$ would be non-trivial.
	Then
	\begin{align*}
	&(1,\ldots,1,g,1,\ldots,1)(h_1,\ldots,h_n)(1,\ldots,1,g,1,\ldots,1)^{-1}(h_1,\ldots,h_n)^{-1}= \\
	& =(1,\ldots,1,gh_ig^{-1}h_i^{-1},1,\ldots,1)
	\end{align*} is a non-trivial element of $\mathcal{N}\cap (1\times\ldots \times 1\times \G\times 1\times \ldots\times 1)$. 
	By the simplicity of $\G$ we find $1\times\ldots \times 1\times \G\times 1\times \ldots\times 1\subseteq \mathcal{N}$.
	
	As $N$ is a proper subgroup of $G$, there exists an $i\in\N$ such that $N[i]$ is a proper subgroup of $\G^{i+1}$. Let $i$ be minimal with this property. Because $N$ is normal in $G$, $N[i]$ is normal in $\G^{i+1}$. By the minimality of $i$ we have $\G^i\times \{1\}\subseteq N[i]$. It follows that $N[i]=\G^i\times\{1\}$. But then necessarily $N=\G^i\times 1\times 1\times \ldots\leq \G^\N$. The surjective morphism
	$$\G^\N\to \G^\N,\ (g_0,g_1,\ldots)\mapsto (g_i,g_{i+1},\ldots)$$ induces an isomorphism $G/N\simeq \G^\N$.
	
	We now treat the case that $\G$ is abelian. Then $\G$ is cyclic of prime order and does not have any proper non-trivial subgroups. For $i\geq 1$ let $\mathcal{N}_i$ denote the subgroup of $\G$ such that the kernel of $\pi_i\colon N[i]\to N[i-1]$ is of the form $\{1\}^i\times\mathcal{N}_i\leq \G^{i+1}$. We also set $\mathcal{N}_0=N[0]$. Then the $\mathcal{N}_i$ are a decreasing chain of subgroups of $\G$ (cf. proof of Proposition~\ref{prop: group shift is of finite type}).
	We cannot have $\mathcal{N}_i=\G$ for all $i\in \N$ because this would imply $N=G$.

	 Thus there exist an $n\in \N$ such that $\mathcal{N}_0,\ldots,\mathcal{N}_{n-1}$ are all equal to $\G$ and $\mathcal{N}_n,\mathcal{N}_{n+1},\ldots$ are all equal to the trivial group. So $\ld(N)=1$ and $N$ is $n$-step by Lemma \ref{lemma: ld}. In particular, $N$ is finite by Lemma \ref{lemma: finite if and only if ld 1}. As $\G^{n+1}/N[n]$ has the same (prime) cardinality as $\G$ we see that $\G^{n+1}/N[n]\simeq \G$.
	
	The map $G\to (\G^{n+1}/N[n])^\N,\ (g_0,g_1,\ldots)\mapsto (\overline{(g_i,g_{i+1},\ldots,g_{i+n})})_{i\in\N}$ is a surjective morphism of expansive profinite groups with kernel $N$ and therefore induces an isomorphism $G/N\simeq (\G^{n+1}/N[n])^\N\simeq \G^\N$.
\end{proof}

\begin{ex}
	In the proof of Lemma \ref{lemma: normal subgroup of shift on simple group} we have seen that if $\G$ is a finite non-abelian simple group, then every proper normal expansive subgroup $N$ of $\G^\N$ is of the form $N=\G^i\times 1\times 1\ldots\leq \G^\N$ for some $i\in\N$. If $\G$ is not abelian, this is not true anymore, for example, for $\G$ abelian, the ``diagonal'' subgroup $N=\{(g,g,\ldots)|\ g\in\G\}$ is a proper normal expansive subgroup.
\end{ex}
One can recover $\G$ from $\G^\N$:

\begin{lemma} \label{lemma: recover group}
	Let $\G$ and $\H$ be finite groups. If the full one-sided group shifts $\G^\N$ and $\H^\N$ are isomorphic, then $\G$ and $\H$ are isomorphic.
\end{lemma}
\begin{proof}
	It suffices to note that $\G$ can be recovered from $G=\G^\N$ as the kernel of $\s\colon G\to G$.
\end{proof}

\subsection{$\s$-Infinitesimal expansive profinite groups}

In this short subsection we deal with the groups that occur in the last position in the subnormal series in the main decomposition theorem (Theorem \ref{theo: main}).

\begin{defi}
	An expansive profinite group $G$ is \emph{$\s$-infinitesimal} if for every $g\in G$ there exists an $n\in\N$ such that $\s^n(g)=1$.
\end{defi}

\begin{ex}
	In the proof of Lemma \ref{lemma: normal subgroup of shift on simple group} we have seen that every proper normal expansive subgroup of a full one-sided shift on a finte non-abelian simple group is $\s$-infinitesimal.
\end{ex}

\begin{lemma} \label{lemma: sinfinitesimal}
	Let $G$ be an expansive profinite group. Then $G$ is $\s$-infinitesimal if and only if $G$ is finite and some power of $\s\colon G\to G$ is the trivial endomorphism $g\mapsto 1$. 
\end{lemma}
\begin{proof}
	By Lemma \ref{lemma: exists standard embedding}, we may assume that $G=G_\Gamma$ for some directed group graph $\Gamma$.
	Suppose $\Gamma$ contains a cycle that is not equal to the cycle whose only edge is $(1,1)$. Looping inside this cycle yields a periodic point $g\in G$. This contradicts the assumption that $\s^n(g)=1$ for some $n\in \N$.	Thus the only circuit in $\Gamma$ is stationary at $1$. This implies that $G$ is finite and so $\s^n(g)=1$ for all $g\in G$ for $n\gg 1$.
	
	The reverse implication is clear. 
\end{proof}

Note that a $\s$-infinitesimal expansive profinite group is $\s$-connected because every $\s$-stable subset contains $1$. For finite groups there is a converse:

\begin{lemma} \label{lemma: finite and sconnected implies infinitesimal}
	A finite $\s$-connected expansive profinite group is $\s$-infinitesimal.
\end{lemma}
\begin{proof}
Let $G$ be a finite $\s$-connected expansive group. Then $G$ is the disjoint union of the $\s$\=/closed sets $\{g\in G|\ \exists \ n\in\N \colon \s^n(g)=1\}$ and $\{g\in G|\ \s^n(g)\neq 1 \ \forall \ n\in \N\}$. The former set is non-empty because it contains $1$ and must therefore equal $G$.
\end{proof}

\section{The decomposition theorem}

In this section we prove our main result: the decomposition theorem (Theorem \ref{theo: main}).
The proof proceeds by induction on the limit degree. We first tackle the case where the induction hypothesis cannot be applied. More precisely, we show that for an infinite expansive $\s$-connected profinite group $G$ such that $\ld(N)\in\{1,\ld(G)\}$ for any normal expansive subgroup $N$ of $G$, there exists an $\ell\in\N$ such that $G/\ker(\s^\ell)$ is isomorphic to a full one-sided group shift on a finite simple group.

\medskip

Let $G$ be an expansive profinite group. It will be useful to consider the set $\operatorname{Emb}(G)$ of all morphisms $\f\colon G\to \G^\N$ from $G$ to a full one-sided group shift such that 
\begin{itemize}
	\item $\ker(\f)$ agrees with the kernel of $\s^\ell\colon G\to G$ for some $\ell\in\N$, in particular, $\ker(\f)$ is $\s$-infinitesimal (for $\ell=0$, by definition, $\s^\ell=\id$ and so $\ker(\s^\ell)=1$), 
	\item  $\f(G)[0]=\G$ and
	\item $\ld(G)=|\ker(\f(G)[1]\xrightarrow{\pi_1}\f(G)[0])|$.
\end{itemize}	
Recall that $\pi_1$ and the notation $G[i]$ was defined in Section \ref{sec: onesided group shift}. Note that by Lemma~\ref{lemma: exists standard embedding} $\operatorname{Emb}(G)$ is non-empty. Moreover, for $\f\in\operatorname{Emb}(G)$ we have $\ld(\f(G))=\ld(G)$ and $\f(G)$ is $1$-step (Lemmas~\ref{lemma: ld multiplicative on quotients}, \ref{lemma: finite if and only if ld 1}, \ref{lemma: sinfinitesimal} and \ref{lemma: ld}).

\begin{lemma} \label{lemma: for Babbitt step1 sr}
	Let $G$ be an expansive profinite group and let $\f\colon G\to \G^\N$ be an element of $\operatorname{Emb}(G)$ such that $|\G|$ is minimal. Then, for every $i\in\N$, the map $\f(G)\to \G,\ (g_0,g_1,\ldots)\mapsto g_i$ is surjective.
\end{lemma}
\begin{proof}
	Assume, for a contradiction, that there exists an $i\in \N$ such that the image $\H\leq \G$ of  $\f(G)\to \G,\ (g_0,g_1,\ldots)\mapsto g_i$ is properly contained in $\G$. 
	The map $$\f'\colon \f(G)\to \H^\N,\ (g_0,g_1,\ldots)\mapsto (g_i,g_{i+1},\ldots)$$ is a morphism of expansive profinite groups and so is the composition $\f''=\f'\f\colon G\to \H^\N$. We will show that $\f''\in\operatorname{Emb}(G)$.

	Assume $\ker(\f)=\ker(\s^\ell)$. We claim that $\ker(\f'')=\ker(\s^{\ell+i})$. If $g\in\ker(\f'')$, then $\f(g)\in\ker(\f')$ and so $\s^i(\f(g))=1$. Thus $\f(\s^i(g))=\s^i(\f(g))=1$ and so $\s^i(g)\in \ker(\f)=\ker(\s^\ell)$. Therefore $g\in\ker(\s^{\ell+i})$.
	
	Conversely, if $g\in\ker(\s^{\ell+i})$, then $\s^i(g)\in\ker(\s^\ell)=\ker(\f)$, and so $\s^i(\f(g))=\f(\s^i(g))=1$. Thus $\f(g)\in\ker(\f')$ and $g\in\ker(\f'')$.
	
	By construction $\f''(G)[0]=\H$. We have a commutative diagram 
	$$
	\xymatrix{
		\f''(G)[1] \ar_{\rho_1}[d] \ar[r] & \f''(G)[0] \ar^{\rho_0}[d] \\
		\f(G)[1] \ar[r] & \f(G)[0]	
	}
	$$
	where $\rho_1(h_0,h_1)=(h_0,h_1)$ and $\rho_0(h_0)=h_0$. As	$\rho_1$ maps the kernel of $\f''(G)[1] \to \f''(G)[0]$ injectively into the kernel of $\f(G)[1] \to \f(G)[0]$, we see that
	$$|\ker(\f''(G)[1] \to \f''(G)[0])|\leq |\ker(\f(G)[1] \to \f(G)[0])|=\ld(G),$$
	where the latter equality follows from $\f\in\operatorname{Emb}(G)$.
	By Lemma \ref{lemma: ld} the sequence $|\ker(\f''(G)[j] \to \f''(G)[j-1])|_{j\geq 1}$ is non-increasing and stabilizes with value $\ld(\f''(G))$. But, using Lemmas \ref{lemma: ld multiplicative on quotients} and \ref{lemma: finite if and only if ld 1}, we find $$\ld(\f''(G))=\ld(G/\ker(\f''))=\frac{\ld(G)}{\ld(\ker(\f''))}=\ld(G).$$ In summary, it follows that $\f''\in\operatorname{Emb}(G)$. Because $|\H|<|\G|$, this contradicts the choice of~$\f$.
\end{proof}

\begin{prop} \label{prop: if no normal then benign}
	Let $G$ be an infinite expansive profinite group and let $\f\colon G\to \G^\N$ be an element of $\operatorname{Emb}(G)$ such that $|\G|$ is minimal. Then there exists a normal non-trivial subgroup $\mathcal{N}$ of $\G$ such that $\mathcal{N}^\N\subseteq\f(G)$.
	%
\end{prop}
\begin{proof}
	%
	%
	%
	We consider $\f(G)[1]\leq \G\times\G$ and $\pi_1\colon \f(G)[1]\to\f(G)[0]=\G,\ (g_0,g_1)\mapsto g_0$. We also have a group homomorphism $\s_1\colon \f(G)[1]\to \G,\ (g_0,g_1)\mapsto g_1$. 
	Let $\G_1\leq \G$ be such that $\ker(\pi_1)=\{1\}\times\G_1$. Similarly, let $\G'\leq \G$ be such that $\ker(\s_1)=\G'\times 1$. Then $\G'\times \G_1$ is a normal subgroup of $\f(G)[1]$.
	Moreover, since $\s_1$ is surjective (Lemma \ref{lemma: for Babbitt step1 sr}), $\G_1$ is normal in $\G$ and because $\pi_1$ maps onto $\G$, $\G'$ is normal in $\G$. Thus $\mathcal{N}=\G_1\cap\G'$ is normal in $\G$. 
	
	Suppose $\mathcal{N}=1$. Define
	$\H=\f(G)[1]/(\G'\times \G_1)$ and 
	%
	%
	consider the morphism $$\f'\colon G\to \H^\N,\ g\mapsto (\overline{\f(g)_i,\f(g)_{i+1}})_{i\in\N}$$ of expansive profinite groups. Note that $\f'$ is the composition of $\f\colon G\to \f(G)$, the $2$-block presentation $\f(G)\simeq (\f(G)[1])^\N$ and the $1$-block map $(\f(G)[1])^\N\to\H^\N$.
	
	We will show that $\f'\in\operatorname{Emb}(G)$.
	Let $\ell\in\N$ be such that $\ker(\f)=\ker(\s^\ell)$. We claim that $\ker(\f')=\ker(\s^{\ell+1})$. If $g\in\ker(\f')$, then $\f(g)$ lies in the kernel of the map 
	$$\f(G)\to \H^\N,\ (g_0,g_1,\ldots)\mapsto \left(\overline{(g_i,g_{i+1})}\right)_{i\in\N},$$
	which equals $\G'\times 1\times 1\ldots\leq \G^\N$ because $\mathcal{N}=1$. 
	Thus $\f(\s(g))=\s(\f(g))=1$. So $\s(g)\in\ker(\f)=\ker(\s^\ell)$ and therefore $g\in\ker(\s^{\ell+1})$. 
	
	Conversely, if $g\in\ker(\s^{\ell+1})$, then $\s(g)\in\ker(\s^\ell)=\ker(\f)$ and so $\s(\f(g))=\f(\s(g))=1$. Thus $\f(g)=(g',1,1,\ldots)$ with $g'\in\G'$ and therefore $g\in\ker(\f')$.

	In particular, $\ker(\f')$ is $\s$-infinitesimal and therefore finite (Lemma \ref{lemma: sinfinitesimal}).
	So, using Lemmas~\ref{lemma: ld multiplicative on quotients} and \ref{lemma: finite if and only if ld 1} we have  $$\ld(\f'(G))=\ld(G/\ker(\f'))=\frac{\ld(G)}{\ld(\ker(\f'))}=\ld(G).$$
	The surjective group homomorphism $\f(G)[1]\to \G/\G',\ (g_0,g_1)\mapsto \overline{g_0}$ has kernel $\G'\times \G_1$ and therefore induces an isomorphism $\eta\colon \H\to \G/\G'$. The group homomorphism
	$$\xi\colon \f(G)[1]\to \H\times\H,\ (g_0,g_1)\mapsto \left(\overline{(g_0,g_1)}, \eta^{-1}(\overline{g_1})\right)$$ has image  $\f'(G)[1]$, because an element of $\f'(G)[1]$ is of the form $\left(\overline{(g_0,g_1)},\overline{(g_1,g_2)}\right)=\left(\overline{(g_0,g_1)},\eta^{-1}(\overline{g_1})\right)$ with $(g_0,g_1,g_2)\in\f(G)[2]$, i.e., $(g_0,g_1)\in\f(G)[1]$ and $(g_1,g_2)\in\f(G)[1]$ as $\f(G)$ is $1$-step.
	The kernel of $\xi$ is $\G'\times(\G'\cap\G_1)=\G'\times 1$. It follows that $|\f'(G)[1]|=\frac{|\f(G)[1]|}{|\G'|}$ and so $$|\ker(\f'(G)[1]\to\f'(G)[0])|=\frac{|\f'(G)[1]|}{|\f'(G)[0]|}=\frac{\frac{|\f(G)[1]|}{|\G'|}}{\frac{|\f(G)[1]|}{|\G'|\cdot|\G_1|}}=|\G_1|=\ld(G),$$
	where the last equality above holds because $\f\in\operatorname{Emb}(G)$. 
	Since $\f'(G)[0]=\H$ by construction, we see that $\f'\in\operatorname{Emb}(G)$.
	
	By the minimality of $|\G|$ we have $|\H|\geq|\G|$. But $\H\simeq \G/\G'$ and so we must have $\G'=1$. By Lemma \ref{lemma: for Babbitt step1 sr} the map $\f(G)[1]\to \G,\ (g_0,g_1)\mapsto g_1$ is surjective. As $\G'=1$ it is an isomorphism. So $|\f(G)[1]|=|\G|$. But also $|\f(G)[0]|=|\G|$ and therefore $$\ld(G)=|\G_1|=\frac{|\f(G)[1]|}{|\f(G)[0]|}=1.$$
	By Lemma \ref{lemma: finite if and only if ld 1} this contradicts the assumption that $G$ is infinite. Thus $\mathcal{N}\neq 1$. 
	Since $\mathcal{N}\times\mathcal{N}\subseteq\f(G)[1]$ and $\f(G)$ is $1$-step, we see that $\mathcal{N}^\N\subseteq\f(G)^\N$.	
\end{proof}

The following corollary is a key step in our proof of the decomposition theorem. 

\begin{cor} \label{cor: if no normal then benign}
	Let $G$ be an infinite $\s$-connected expansive profinite group such that for every normal expansive subgroup $N$ of $G$ we have $\ld(N)=1$ or $\ld(N)=\ld(G)$. Then there exists an $\ell\in\N$ such that $G/\ker(\s^\ell)$ is isomorphic to a full one-sided group shift on a finite simple group.
\end{cor}
\begin{proof}
	We continue to use the notation of the proof of Proposition \ref{prop: if no normal then benign}. In particular, $\f\colon G\to\G^\N$ is an element of $\operatorname{Emb}(G)$ such that $|\G|$ is minimal. We will show that $\f$ is surjective and that $\G$ is simple.
	
	In the proof of Proposition \ref{prop: if no normal then benign} we have seen that $\mathcal{N}=\G'\cap \G_1$ is non-trivial. As $\mathcal{N}^\N\subseteq\f(G)$ maps to $1$ under $$\f(G)\to \H^\N,\ (g_0,g_1,\ldots)\mapsto \left(\overline{(g_i,g_{i+1})}\right)_{i\in\N},$$
	we deduce that $\f^{-1}(\mathcal{N}^\N)\subseteq\ker(\f')$. In particular, $N=\ker(\f')$ is infinite and therefore has limit degree strictly greater than $1$ (Lemma \ref{lemma: finite if and only if ld 1}).
	As $N$ is a normal expansive subgroup of $G$ we have, by assumption, $\ld(N)=\ld(G)$.

Thus $\ld(G/N)=1$ and therefore $G/N\simeq \f'(G)$ is finite (Lemma~\ref{lemma: finite if and only if ld 1}). Because $G$ is $\s$\=/connected, also $G/N$ is $\s$-connected (Lemma~\ref{lemma: quotient of sconnected is sconnected}). So $G/N$ is a finite $\s$-connected expansive profinite group and must therefore be $\s$-infinitesimal by Lemma \ref{lemma: finite and sconnected implies infinitesimal}. Thus there exists an $n\in\N$ such that $\s^n(\f'(G))=1$, i.e., $(\f(g)_i,\f(g)_{i+1})\in\G'\times \G_1$ for $i\geq n$. So $\f(g)_i\in \G'$ and $\f(g)_{i+1}\in\G_1$ for all $g\in G$ and $i\geq n$. On the other hand, by Lemma~\ref{lemma: for Babbitt step1 sr}, we have $\{\f(g)_i|\ g\in G\}=\G$ for every $i\in \N$. This shows that $\G'=\G$ and $\G_1=\G$. But then $\f(G)[1]=\G\times\G$ and $\f(G)$ is the full one-sided group shift on $\G$. So $G/\ker(\f)\simeq \G^\N$.
		
	 It remains to see that $\G$ is a simple group. Suppose $\G$ has a non-trivial proper normal subgroup $\mathcal{N}$. Then $\mathcal{N}^\N$ is a normal expansive subgroup of $\G^\N$ and $N=\f^{-1}(\mathcal{N}^\N)$ is a normal expansive subgroup of $G$. Since $N/\ker(\f)\simeq \mathcal{N}^\N$ we have
	$$\ld(N)=\frac{\ld(N)}{\ld(\ker(\f))}=\ld(\mathcal{N}^\N)=|\mathcal{N}|$$
	by Lemma \ref{lemma: ld multiplicative on quotients} and Example \ref{ex: ld of full shift}. As $1<|\mathcal{N}|<|\G|=\ld(G)$ we arrive at a contradiction.
\end{proof}

The following Corollary is a one-sided version of the Corollary to Theorem 2 in \cite{Kitchens:ExpansiveDynamicsOnZeroDimensionalGroups}.
\begin{cor}
	Let $G$ be a $\s$-connected expansive profinite group such that $p=\ld(G)$ is a prime number. Then there exists an $\ell\in\N$ such that $G/\ker(\s^\ell)$ is isomorphic to the full one-sided group shift on the finite cyclic group of order $p$.
\end{cor}
\begin{proof}
	As the assumptions of Corollary \ref{cor: if no normal then benign} are met, there exists an $\ell\in\N$ such that $G/\ker(\s^\ell)$ is isomorphic to the full one-sided group shift on a finite simple group $\G$. Because $p=\ld(G)=\ld(G/\ker(\s^\ell)=\ld(\G^\N)=|\G|$, it follows that $\G$ is cyclic of order $p$.
\end{proof}

The following lemma will be useful for the induction step in the proof of the main decomposition theorem. Roughly speaking, it allows us to remove the top $\s$-infinitesimal quotient in a subnormal series.
\begin{lemma} \label{lemma: exchange sinfinitesimal for Babbitt}
	Let $G$ be an expansive profinite group with a subnormal series
	$$G\supseteq G_1\supseteq G_2\supseteq \cdots\supseteq G_{n+1}$$
	such that $G/G_1$ is $\s$-infinitesimal, $G_i/G_{i+1}$ is isomorphic to a full one-sided group shift on a finite simple group for $i=1,\ldots n$ and $G_{n+1}$ is $\s$-infinitesimal. Then there exists a subnormal series
	$$G\supseteq H_1\supseteq H_2\supseteq \cdots\supseteq H_n$$
	such that $G/H_1$ and $H_i/H_{i+1}$ ($i=1,\ldots,n-1$) are isomorphic to full one-sided group shifts on finite simple groups and $H_n$ is $\s$-infinitesimal.
\end{lemma}
\begin{proof}
	Since $G/G_1$ is $\s$-infinitesimal, there exists an $r\in\N$ such that $\s^r(g)=1$ for all $g\in G/G_1$ (Lemma \ref{lemma: sinfinitesimal}). In other words, $\s^{-r}(G_1)=G$. We claim that the subnormal series
	$$G=\s^{-r}(G_1)\supseteq \s^{-r}(G_2)\supseteq\ldots\supseteq \s^{-r}(G_{n+1})$$
	has the required properties. As $\s$ is surjective on a full one-sided group shift we see that $\s^r\colon \s^{-r}(G_i)\to G_i/G_{i+1}$ is surjective and so $\s^{-r}(G_i)/\s^{-r}(G_{i+1})$ is isomorphic to $G_i/G_{i+1}$ for $i=1,\ldots,n$.
	
	Finally, if $s\in\N$ is such that $\s^s(g)=1$ for all $g\in G_{n+1}$, then $\s^{r+s}(g)=1$ for all $g\in \s^{-r}(G_{n+1})$. Thus $\s^{-r}(G_{n+1})$ is $\s$-infinitesimal.
\end{proof}

%
%
%
%
%
%

Finally, we are prepared to prove our main result.
\begin{theo} \label{theo: main}
	Let $G$ be an expansive profinite group. Then there exists a subnormal series
	$$G\supseteq G_1\supseteq G_2\supseteq \cdots\supseteq G_n$$
	such that $G_1=G^\sc$, $G_i/G_{i+1}$ is isomorphic to a full one-sided group shift on a finite simple group $\G_i$ for $i=1,\ldots,n-1$ and $G_n$ is $\s$-infinitesimal. If 
	$$G\supseteq H_1\supseteq H_2\supseteq \cdots\supseteq H_m$$
	is another subnormal series such that $H_1=G^\sc$, $H_i/H_{i+1}$ is isomorphic to a full one-sided group shift on a finite simple group $\H_i$ for $i=1,\ldots,m-1$ and $H_m$ is $\s$-infinitesimal, then $m=n$ and there exists a permutation $\pi$ such that $\G_i$ is isomorphic to $\H_{\pi(i)}$ for $i=1,\ldots,n-1$.
\end{theo}
\begin{proof}
	We first establish the existence of the decomposition by induction on $\ld(G)$. If $\ld(G)=1$, then $G$ is finite (Lemma \ref{lemma: finite if and only if ld 1}) and the theorem holds with $n=1$ by Example~\ref{ex: main for finite}.
	
	So we may assume that $\ld(G)>1$. Replacing $G$ with $G^\sc$, we may also assume that $G$ is $\s$-connected. (Note that by Lemma \ref{lemma: ld for sc} the limit degree remains unchanged.)
If there does not exist a normal expansive subgroup $N$ of $G$ such that $1<\ld(N)<\ld(G)$, then there exists a decomposition of the desired form with $n=2$ by Corollary \ref{cor: if no normal then benign}.
	
So we may assume that there exists a normal expansive subgroup $N$ of $G$ such that $1<\ld(N)<\ld(G)$. We know from Lemma \ref{lemma: Nsc is normal} that also $N^\sc$ is a normal expansive subgroup of $G$. Moreover $\ld(N)=\ld(N^\sc)$ by Lemma \ref{lemma: ld for sc}. Replacing $N$ by $N^\sc$ we may thus assume that $N$ is $\s$-connected.
	
Because $\ld(G/N)=\ld(G)/\ld(N)<\ld(G)$ we can apply the induction hypothesis to $G/N$. As $G$ is $\s$-connected, also $G/N$ is $\s$-connected (Lemma \ref{lemma: quotient of sconnected is sconnected}). So, using Proposition~\ref{prop: isomorphosm theorems}, we obtain a subnormal series
	$$G/N\supseteq G_1/N\supseteq\cdots\supseteq G_n/N$$ for $G/N$,
	where $G\supseteq G_1\supseteq\cdots\supseteq G_n\supseteq N$ is a subnormal series for $G$ such that $G_n/N$ is $\s$-infinitesimal and $(G_i/N)/(G_{i+1}/N)=G_i/G_{i+1}$ is isomorphic to a full one-sided group shift on a finite simple group for $i=0,\ldots,n-1$, where $G_0:=G$.
	
	As $\ld(N)<\ld(G)$, we can also apply the induction hypothesis to $N$. Since $N$ is $\s$-connected, we obtain a subnormal series
	$$N\supseteq N_1\supseteq\cdots\supseteq N_m,$$
	with $N_m$ $\s$-infinitesimal and $N_i/N_{i+1}$ isomorphic to a full one-sided group shift on a finite simple group for $i=0,\ldots,m-1$ ($N_0:=N$). By Lemma~\ref{lemma: exchange sinfinitesimal for Babbitt}, the subnormal series
	$$G_n\supseteq N\supseteq N_1\supseteq\cdots\supseteq N_m$$
	can be replaced by a subnormal series
	$$G_n\supseteq H_1\supseteq\cdots\supseteq H_m,$$
	with $G_n/H_1$ and $H_i/H_{i+1}$ ($i=1,\ldots,m-1$) isomorphic to a full one-sided group shift on a finite simple group and $H_m$ $\s$-infinitesimal. Then
	$$G\supseteq G_1\supseteq\cdots\supseteq G_n\supseteq H_1\supseteq\cdots\supseteq H_m$$ is a subnormal series for $G$ of the required form.
	
	\medskip
	
	We next address the uniqueness: Assume that
	$$G^\sc=G_1\supseteq G_2\supseteq\ldots\supseteq G_n$$
	and 
		$$G^\sc=H_1\supseteq H_2\supseteq\ldots\supseteq H_m$$
	are subnormal series of $G^\sc$ such that $G_i/G_{i+1}$ is isomorphic to a full one-sided group shift on a simple group $\G_i$ ($i=1,\ldots,n-1$), $H_i/H_{i+1}$ is isomorphic to a full one-sided group shift on a finite simple group $\H_i$ ($i=1,\ldots,m-1$) and the groups $G_n$ and $H_m$ are $\s$-infinitesimal.
	
	By Proposition \ref{prop: schreier refinement} these two subnormal series for $G^\sc$ have equivalent refinements. Let
	\begin{equation} \label{eq: refinement 1}
G_1\supseteq G_{1,1}\supseteq G_{1,2}\supseteq\ldots\supseteq G_{1,n_1}=G_2\supseteq \ldots\supseteq G_n\supseteq G_{n,1}\supseteq\ldots\supseteq G_{n,n_n}=1
\end{equation}
	and
	\begin{equation} \label{eq: refinement 2}
H_1\supseteq H_{1,1}\supseteq H_{1,2}\supseteq\ldots\supseteq H_{1,m_1}= H_2\supseteq \ldots\supseteq H_m \supseteq H_{m,1}\supseteq\ldots\supseteq H_{m,m_m}=1
	\end{equation}
	be such refinements. We may assume that all of the above inclusions are proper. Note that $G_{i,1}/G_{i+1}$ is a proper normal expansive subgroup of $G_i/G_{i+1}$ for $i=1,\ldots,n-1$. By Lemma \ref{lemma: normal subgroup of shift on simple group} the group $G_{i,1}/G_{i+1}$ is finite. It follows that all the factor groups $G_{i,j}/G_{i,j+1}$ ($j=1,\ldots,n_i-1$) are finite, whereas $G_i/G_{i,1}$ is infinite. Thus, the number of infinite factor groups of the subnormal series (\ref{eq: refinement 1}) is exactly $n-1$. Similarly, the number of infinite factor groups of the subnormal series (\ref{eq: refinement 2}) is $m-1$. Because the subnormal series (\ref{eq: refinement 1}) and (\ref{eq: refinement 2}) are equivalent, we see that $n=m$. Moreover, by Lemma \ref{lemma: normal subgroup of shift on simple group} the $n-1$ infinite factor groups of (\ref{eq: refinement 1}) are isomorphic to full one-sided group shifts on the finite simple groups $\G_1,\ldots,\G_{n-1}$. Similarly, the $n-1=m-1$ infinite factor groups of (\ref{eq: refinement 2}) are isomorphic to full one-sided group shifts on the finite simple groups $\H_1,\ldots,\H_{m-1}$. The equivalence of (\ref{eq: refinement 1}) and (\ref{eq: refinement 2}) together with Lemma \ref{lemma: recover group} shows that there exists a permutation $\pi$ such that $\G_i\simeq \H_{\pi(i)}$ for $i=1,\ldots,n-1$.	
\end{proof}

\begin{rem} \label{rem: explain theoremA}
	For simplicity, Theorem \ref{theo: main} is stated in the introduction without reference to the $\s$-identitiy component and the claim concerning the uniqueness of the group $G_1$ made there needs some justification: Let $G\supseteq G_1\supseteq\ldots\supseteq G_n$ be a subnormal series for an expansive profinite group $G$ such that $G_i/G_{i+1}$ is isomorphic to a full one-sided group shift on a finite simple group for $i=1,\ldots,n-1$, $G/G_1$ is finite with $\s\colon G/G_1\to G/G_1$ an automorphism and $G_n$ is $\s$-infinitesimal. Then $G_1=G^\sc$.
\end{rem}
\begin{proof}
	Full one-sided group shifts and $\s$-infinitesimal expansive profinite groups are $\s$-connected. So it follows inductively, using Lemma \ref{lemma: sconnected in short exact sequence}, that all the $G_i$ are $\s$-connected. In particular $G_1$ is $\s$-connected. Thus the claim follows from Lemma \ref{lemma: sconnected converse}.
\end{proof}

We next consider some examples that illustrate Theorem \ref{theo: main}. The following example shows that our decomposition theorem can be interpreted as a generalization of the classical Jordan-H\"{o}lder theorem.

\begin{ex}
	Let $G$ be the full one-sided group shift on the finite group $\G$ and let $\G=\G_1\supseteq \G_2\supseteq\ldots\supseteq \G_n=1$ be a decomposition series for $\G$. Then 
	$$G=\G_1^\N\supseteq \G_2^\N\supseteq\ldots\supseteq\G_n^\N=1$$
	is a subnormal series of $G$ with the properties of Theorem \ref{theo: main}. Note that $G$ is $\s$-connected by Example \ref{ex: full shift is sconnected}.
\end{ex}

\begin{ex}
	Let $\G$ be a finite simple group and let $\G'$ be any finite group containing $\G$. Then $G=\G'\times\G\times\G\times\ldots$ is an expansive profinite group under $\s\colon G\to G,\ (g',g_1,g_2,\ldots)\mapsto (g_1,g_2,\ldots)$.
	Moreover, $G$ is $\s$-connected (e.g., by Lemma \ref{lemma: sconnected component} (iii)) and $G_2=\G'\times 1\times 1\ldots\leq G$ is a normal $\s$-infinitesimal expansive subgroup of $G$ such that $G/G_2$ is isomorphic to the full one-sided group shift on $\G$. Thus $G=G_1\supseteq G_2$ is a decomposition as in Theorem~\ref{theo: main}.
\end{ex}

\begin{ex}
	This example is taken from \cite{Kitchens:ExpansiveDynamicsOnZeroDimensionalGroups} (Example 4). However, we use a multiplicative notation, so that subgroups can easily be described by equations (rather than by listing elements). Let us write $\Gm=\mathbb{C}^\times=\mathbb{C}\smallsetminus\{0\}$ for the multiplicative group of the complex numbers and consider $\Gm^\N$ as a group under componentwise multiplication. For $g=(g_0,g_1,\ldots)\in\Gm^\N$ we set $\s(g)=(g_1,g_2,\ldots)$ as usual. Let
	$$G=\left\{(g,h)\in(\Gm^2)^\N|\ g^4=1,\ h^2=1,\ \s(h)=g^2h\right\}.$$
	Then $G$ is a subgroup of $\G^\N$, where $\G=\{a\in\mathbb{C}^\times|\ a^4=1\}\times \{b\in\mathbb{C}^\times|\ b^2=1\}$ is a product of a cyclic group of order four and a cyclic group of order two. Indeed $G$ is a $1$-step group shift on $\G$.
	Set $G_2=\{(g,h)\in G|\ h=1\}=\{(g,1)\in\Gm|\ g^2=1\}$. So $G_2$ is a full one-sided group shift on a cyclic group of order two. 
	The map $G\to \Gm^\N,\ (g,h)\to h$ has kernel $G_2$ and image $\{h\in\Gm^\N|\ h^2=1\}$. Thus $G/G_2$ is isomorphic to a full one-sided group shift on a cyclic group of order two. So there exists a short exact sequence
	\begin{equation} \label{eq: exact sequence}
	1 \to C_2^\N \to G\to C_2^\N\to 1,
	\end{equation} where $C_2$ is the cyclic group of order two.
	By Lemma \ref{lemma: sconnected in short exact sequence} and Example \ref{ex: full shift is sconnected} the expansive profinite group $G$ is $\s$-connected. So $$G=G_1\supseteq G_2\supseteq G_3=1$$ is a subnormal series as in Theorem \ref{theo: main}. We note that the exact sequence (\ref{eq: exact sequence}) is not split. Indeed, $G$ is not isomorphic to a full one-sided group shift (\cite[Example 4]{Kitchens:ExpansiveDynamicsOnZeroDimensionalGroups}).
	
\end{ex}


\begin{ex} We use the same notation as in the previous example.
	Set
	$$G=\big\{(g_1,g_2,g_3)\in (\Gm^3)^\N|\ g_1^4=g_2^4=g_3^2=1,\ \s(g_1)=g_2^2,\ \s(g_3)=g_3\big\}.$$
	Then $G$ is a one-sided group shift on the finite group $\G$, where $\G$ is a direct product of two cyclic groups of order four and a cyclic group of order two. 
	Let $G_1$, $G_2$ and $G_3$ be the subgroups of $G$ given by
	$$G_1=\left\{(g_1,g_2,1)\in (\Gm^3)^\N|\ g_1^4=g_2^4=1,\ \s(g_1)=g_2^2,\right\},$$
		$$G_2=\left\{(g_1,g_2,1)\in (\Gm^3)^\N|\ g_1^4=g_2^2=1,\ \s(g_1)=1,\right\},$$
	and $$G_3=\left\{(g_1,1,1)\in (\Gm^3)^\N|\ g_1^4=1,\ \s(g_1)=1 \right\}.$$
	We will show that
	$$G\supseteq G_1\supseteq G_2\supseteq G_3$$
	is a subnormal series as in Theorem \ref{theo: main}.

	Clearly $G_3$ is $\s$-infinitesimal. The map $G_1\to \Gm^\N,\ (g_1,g_2,1)\mapsto g_2$ has kernel $G_3$ and image $\{g\in\Gm^\N|\ g^4=1\}$, a full one-sided group shift on a cyclic group of order four. The full one-sided group shift $\{g\in\Gm^\N|\ g^2=1\}$ contained in the image corresponds to $G_2$. So $G_1/G_2$ and $G_2/G_3$ are both full one-sided group shifts on a cyclic group of order two. 
%
	
	So it only remains to show that $G_1=G^\sc$. As in Remark \ref{rem: explain theoremA} it follows that $G_1$ is $\s$-connected. According to Lemma \ref{lemma: sconnected converse} it suffices to show that $G/G_1$ is finite with $\s\colon G/G_1\to G/G_1$ bijective. But $G/G_1$ is a group of order two with $\s$ the identity map.
\end{ex}

\section{Topological conjugacy}

Recall that two topological spaces $(X,\s)$ and $(Y,\s)$ equipped with continuous endomorphisms are \emph{topologically conjugate} if there exists a homeomorphism $X\to Y$ such that
$$
\xymatrix{
	X \ar[r] \ar_\s[d] & Y \ar^\s[d] \\
	X \ar[r] & Y
}
$$
commutes. In \cite{Kitchens:ExpansiveDynamicsOnZeroDimensionalGroups} B. Kitchens showed that every profinite group equipped with an expansive \emph{automorphism} is topologically conjugate to $\mathcal{A}^\Z\times \mathcal{F}$, a full two-sided shift on a finite set $\mathcal{A}$ times a finite (discrete) set $\mathcal{F}$ equipped with an automorphism.

%
%
%
%
%

We establish here a similar result for expansive \emph{endomorphisms}: If $G$ is an expansive profinite group, then there exists an $\ell\in\N$ such that $\s^\ell(G)$ is topologically conjugate to $\mathcal{A}^\N\times\mathcal{F}$, a full one-sided shift on a finite set $\mathcal{A}$ times a finite set $\mathcal{F}$ equipped with an automorphism. 
%
%
%
%
%
We will need the following preparatory lemma.

\begin{lemma} \label{lemma: split}
	Let $G$ be an expansive profinite group and $N$ a normal expansive subgroup of $G$ that is isomorphic to a full one-sided group shift. Then $G$ is topologically conjugate to $G/N\times N$.
\end{lemma}
\begin{proof}
	To be clear, the topology on $G/N\times N$ is the product topology and the endomorphism $\s$ on $G/N\times N$ sends $(h,n)\in G/N\times N$ to $(\s(h),\s(n))$. 
	
	To begin with, choose a continuous section $\varphi$ of the canonical map $\pi\colon G\to G/N$, i.e., a continuous map $\varphi\colon G/N\to G$ such that $\pi\varphi=\id_{G/N}$. Such a map always exists by \cite[Proposition 2.2.2]{RibesZaleskii:ProfiniteGroups}. (Note that we do not require that $\varphi$ commutes with $\s$ or is a group homomorphism.) The map $\eta\colon G/N\times N\to G,\ (h,n)\mapsto\varphi(h)n$ is a homeomorphism with inverse $\eta^{-1}\colon G\to G/N\times N,\ g\mapsto (\pi(g), \varphi(\pi(g))^{-1}g)$. For $(h,n)\in G/N\times N$ we have 
	\begin{align*}
	\eta^{-1}(\s(\eta(h,n))) & =\eta^{-1}(\s(\varphi(h))\s(n))=(\s(\pi(\varphi(h)),\varphi(\pi(\s(\varphi(h))\s(n)))^{-1}\s(\varphi(h))\s(n))=\\
	& =(\s(h), \varphi(\s(h))^{-1}\s(\varphi(h))\s(n)).
	\end{align*}
	Thus $G$ is topologically conjugate to $(G/N\times N, \s')$ with $\s'\colon G/N\times N\to G/N\times N$ given by $\s'(h,n)=(\s(h),\psi(h)\s(n))$, where $\psi\colon G/N\to N,\ h\mapsto \varphi(\s(h))^{-1}\s(\varphi(h))$.
	
	So it suffices to show that $(G/N\times N,\s')$ and $(G/N\times N,\s)$ are topologically conjugate. We may assume that $N=\mathcal{N}^\N$ for some finite group $\mathcal{N}$. So for $h\in G/N$ the element $\psi(h)=(\psi(h)_i)_{i\in\N}$ is a sequence in $\mathcal{N}$. Define a continuous map $\alpha\colon G/N\to N=\mathcal{N}^\N$ by $\alpha(h)_0=1$ and $\alpha(h)_i=\psi(h)^{-1}_{i-1}\psi(\s(h))^{-1}_{i-2}\ldots\psi(\s^{i-1}(h))_0^{-1}$ for $i\geq 1$. Then $\psi(h)_i\alpha(h)_{i+1}=\alpha(\s(h))_i$ for all $i\in\N$, i.e., $\psi(h)\s(\alpha(h))=\alpha(\s(h))$ for all $h\in G/N$.
	
	The map $\xi\colon G/N\times N\to G/N\times N,\ (h,n)\mapsto (h,\alpha(h)n)$ is a homeomorphism and $$\xi(\s(h,n))=(\s(h),\alpha(\s(h))\s(n))=(\s(h), \psi(h)\s(\alpha(h))\s(n))=\s'(h, \alpha(h)n)=\s'(\xi(h,n))$$
	for all $(h,n)\in G/N\times N$. Thus $\xi$ is a conjugacy between $(G/N\times N,\s)$ and $(G/N\times N, \s')$.
\end{proof}

Let $G$ be an expansive profinite group. Note that for $\ell\in\N$, the kernel $\ker(\s^\ell)$ of $\s^\ell\colon G\to G$ is a $\s$-infinitesimal expansive subgroup of $G$. In particular, $\ker(\s^\ell)$ is finite. Moreover $\s^\ell(G)$ is an expansive subgroup of $G$ and $\s^\ell$ induces an isomorphism $G/\ker(\s^\ell)\to\s^\ell(G)$ of expansive profinite groups.

\begin{theo} \label{theo: topological conjugacy}
	Let $G$ be an expansive profinite group. Then there exists an $\ell\in \N$ such that $\s^\ell(G)$ is topologically conjugate to $\mathcal{A}^\N\times \mathcal{F}$, where $\mathcal{A}$ is a finite set and $\mathcal{F}$ is a finite set equipped with a bijective map $\s\colon \mathcal{F}\to \mathcal{F}$ having a fixed point.
\end{theo}
\begin{proof}
	We will prove the theorem by induction on $\ld(G)$. If $\ld(G)=1$, then $G$ is finite (Lemma~\ref{lemma: finite if and only if ld 1}) and for large enough $\ell$, the map $\s\colon \s^\ell(G)\to\s^\ell(G)$ is bijective. Thus the theorem holds with $\mathcal{A}$ a one-element set and $\mathcal{F}=\s^\ell(G)$. (The identity element $1\in\mathcal{F}$ is a fixed point.)
	
	Assume that $\ld(G)>1$. By Proposition \ref{prop: if no normal then benign} there exists an $\ell\in\N$ and a normal expansive subgroup $N$ of $\s^\ell(G)$ such that $N$ is isomorphic to a full one-sided group shift on a non-trivial finite group. From Lemma \ref{lemma: split} we obtain that $\s^\ell(G)$ is topologically conjugate to $\s^\ell(G)/N\times N$. We have $\ld(\s^\ell(G)/N)=\frac{\ld(\s^\ell(G))}{\ld(N)}<\ld(G)$ and so we can apply the induction hypothesis to $G'=\s^\ell(G)/N$: There exists an $\ell'\in\N$, a finite set $\mathcal{A}'$ and a finite set $\mathcal{F}'$ equipped with an automorphism having a fixed point such that $\s^{\ell'}(G')$ is topologically conjugate to $\mathcal{A}'^\N\times\mathcal{F}'$.
	
	Since $\s^\ell(G)$ is topologically conjugate to $G'\times N$, we see that $\s^{\ell+\ell'}(G)$ is topologically conjugate to $\s^{\ell'}(G'\times N)=\s^{\ell'}(G')\times\s^{\ell'}(N)\simeq\mathcal{A}'^\N\times \mathcal{F}'\times N$. So, if $N\simeq \mathcal{N}^\N$, then $\s^{\ell+\ell'}(G)$ is topologically conjugate to $(\mathcal{A}'\times\mathcal{N})^\N\times\mathcal{F}'$.
\end{proof}

	Note that a full one-sided group shift is $\s$-connected and has a surjective $\s$. The following corollary provides a converse for expansive profinite groups:

\begin{cor} \label{cor: expansive group is full shift}
	Let $G$ be a $\s$-connected expansive profinite group with $\s\colon G\to G$ surjective. Then $G$ is topologically conjugate to a full one-sided shift. 
\end{cor}
\begin{proof}
	Since $\s\colon G\to G$ is surjective, it follows from Theorem \ref{theo: topological conjugacy} that $G$ is topologically conjugate to $\mathcal{A}^\N\times \mathcal{F}$. Let $f\in\mathcal{F}$ be a fixed point. Then $\mathcal{A}\times \mathcal{F}$ is the disjoint union of the $\s$-closed sets $\mathcal{A}\times \{f\}$ and $\mathcal{A}\times (\mathcal{F}\smallsetminus\{f\})$. As $G$ is $\s$-connected, we must have $\mathcal{F}=\{f\}$. Thus $G$ is topologically conjugate to $\mathcal{A}^\N$.
\end{proof}

Because full one-sided shifts are $\s$-irreducible, Corollary \ref{cor: expansive group is full shift} implies that a $\s$-connected expansive profinite group $G$ with $\s\colon G\to G$ surjective is $\s$-irreducible.

%
%
%
%
%
%
%
%
%
%
%

\section{Expansive automorphisms}

In this section we establish an analog of Theorem \ref{theo: main} for expansive automorphisms in place of expansive endomorphisms. There are no immediate implications between results about profinite groups equipped with an expansive endomorphism and results about profinite groups equipped with an expansive automorphism. Indeed, if $G$ is a profinite group and $\s\colon G\to G$ is a map that is simultaneously an expansive endomorphism and an expansive automorphism, then $G$ is finite by Corollary \ref{cor: s injective implies finite}. However, it seems possible to obtain a proof of an analog of Theorem \ref{theo: main} for expansive automorphisms by carefully going through all the steps of the proof of Theorem \ref{theo: main} and performing some minor modifications here and there. 
Indeed, the situation is simpler in the automorphism setting because there are no $\s$-infinitesimal groups in this context.

There is, however, a slightly more elegant path that we will follow here. There is a universal construction $G\rightsquigarrow G^*$ that associates to any profinite group $G$ equipped with an expansive endomorphism, a profinite group $G^*$ equipped with an expansive automorphism. Moreover, any profinite group equipped with an expansive automorphism is of the form $G^*$ for some $G$. In this fashion, results about profinite groups with expansive endomorphisms can be transformed to results about profinite groups with expansive automorphisms. This way we are able to avoid having to enter into the details of the proof of the existence part of Theorem \ref{theo: main} again. 

\medskip

We begin by recalling the two-sided setup in symbolic dynamics. See \cite{Kitchens:SymbolicDynamics} or \cite{LindMarcus:IntroductionToSymbolicDynamisAndCoding}. Let $\mathcal{A}$ be a finite set. We consider $\mathcal{A}^\Z$ as a topological space via the product topology of the discrete topology on $\mathcal{A}$. The topological space $\mathcal{A}^\Z$ together with the homeomorphism $\s\colon \mathcal{A}^\Z\to  \mathcal{A}^\Z$, given by $\s(a_{n\in\Z})=(a_{n+1})_{n\in\Z}$ is the \emph{full two-sided shift} on the alphabet $\mathcal{A}$.
A \emph{two-sided shift} on $\mathcal{A}$ is a closed subset $X$ of $\mathcal{A}^\Z$ such that $\s(X)=X$.
A \emph{word} or \emph{block} of length $i$ is a sequence of $i$ elements from $\mathcal{A}$. A two-sided shift $X$ on $\mathcal{A}$ is a (two-sided) \emph{subshift of finite type} if there exists a finite set $\mathcal{F}$ of blocks such that $X$ consists of all elements of $\mathcal{A}^\Z$ that do not contain any blocks from $\mathcal{F}$. If $\Gamma$ is a directed graph with set of vertices equal to $\mathcal{A}$, then the set $X^*_\Gamma\subseteq\mathcal{A}^\Z $ consisting of all biinfinite sequences in $\mathcal{A}$ that trace out a biinfinite directed path in $\Gamma$ is a subshift of finite type.

In case the alphabet $\mathcal{A}=\G$ is a finite group, $\G^\Z$ inherits a group structure. In fact, $\G^\Z$ is a profinite group and $\s\colon \G^\Z\to \G^\Z$ is an automorphism (of profinite groups). A \emph{two-sided group shift} $G$ on $\G$ is a two-sided shift on $\G$ that is a subgroup of $\G^\Z$. In particular, $G$ is a profinite group and $\s\colon G\to G$ is an autmorphism. It is shown in \cite{Kitchens:ExpansiveDynamicsOnZeroDimensionalGroups} that every two-sided group shift is a subshift of finite type. If $\Gamma$ is a directed group graph on a finite group $\G$, then $G_\Gamma^*=X_\Gamma^*$ is a two-sided group shift.

\medskip

In this section we consider expansive endomorphisms and expansive automorphisms of profinite groups. To have a clear notational distinction between the two, we add a ``$\ast$'' to the notation whenever we are dealing with expansive automorphisms.
We continue to use the notation of the previous sections. In particular, an expansive profinite group is a profinite group together with an expansive endomorphism (Definition \ref{defi: expansive profinite group}).

\begin{defi} \label{defi: astexpansive group}
	An automorphism $\s\colon G\to G$ of a profinite group $G$ is an \emph{expansive automorphism} if there exists a neighborhood $U$ of $1$ such that $\cap_{n\in\Z}\s^n(U)=1$. A \emph{$\ast$expansive profinite group} is a profinite group $G$ together with an expansive automorphism $\s\colon G\to G$.
\end{defi}
As for expansive profinite groups, once can assume that $U$ is an open normal subgroup of $G$. 
The study of $\ast$expansive profinite groups was initiated by B. Kitchens in \cite{Kitchens:ExpansiveDynamicsOnZeroDimensionalGroups}.

A \emph{topological $\s$-group} is a topological group $G$ equipped with an endomorphism (i.e., a continuous group homomorphism) $\s\colon G\to G$. A morphism $G\to H$ of topological $\s$-groups is a continuous group homomorphism such that 
$$
\xymatrix{
	G \ar[r] \ar_\s[d] & H \ar^\s[d] \\
	G \ar[r] & H	
}
$$
commutes. A morphism of $\ast$expansive profinite groups is a morphism of topological $\s$\=/groups.

We will need the $\ast$-analogs of the elementary results from Section \ref{subsec: Group theory for expansive profinite groups}. The proofs are very similar to Section \ref{subsec: Group theory for expansive profinite groups}. We therefore omit the details.

\begin{lemma} \label{lemma: astexpansive subgroups and quotients}
	Let $G$ be a $\ast$expansive profinite group.
	\begin{enumerate}
		\item If $H$ is a closed subgroup of $G$ such that $\s(H)=H$, then $H$ (with the induced topology and automorphism) is a $\ast$expansive profinite group. In this case, we call $H$ a $\ast$expansive subgroup of $G$.
		\item If $N$ is a normal $\ast$expansive subgroup of $G$, then $G/N$ (with the quotient topology and induced automorphism) is a $\ast$expansive profinite group and the canonical map $G\to G/N$ is a morphism of $\ast$expansive profinite groups. 
	\end{enumerate}	
\end{lemma}
We note that point (ii) of Lemma \ref{lemma: astexpansive subgroups and quotients} is proved in far greater generality in \cite{GloecknerRaja:ExpansiveAutomorphismsOfTotallyDisconnectedLocallyCompactGroups}.

\begin{prop}[Isomorphism theorems for $\ast$expansive profinite groups] \label{prop: astisomorphosm theorems} \mbox{} 
	\begin{enumerate} 
		\item Let $\f\colon G\to H$ be a morphism of $\ast$expansive profinite groups. Then $\f(G)$ is a $\ast$expansive subgroup of $H$, $\ker(\f)$ is a normal $\ast$expansive subgroup of $G$ and  the canonical map $G/\ker(\f)\to \f(G)$ is an isomorphism of $\ast$expansive profinite groups.
		\item Let $N$ be a normal $\ast$expansive subgroup of a $\ast$expansive group $G$ and $\pi\colon G\to G/N$ the canonical map. Then the map
		$$\{\text{$\ast$expansive subgroups of $G$ containing $N$}\}\longrightarrow\{\text{$\ast$expansive subgroups of $G/N$}\},$$ $H\mapsto \pi(H)=H/N$
		is a bijection with inverse $H'\mapsto \pi^{-1}(H')$. Moreover $H$ is normal in $G$ if and only if $H/N$ is normal in $G/N$ and in that case $G/H\simeq (G/N)/(H/N)$.
		\item Let $H$ and $N$ be $\ast$expansive subgroups of a $\ast$expansive profinite group $G$ such that $H$ normalizes $N$. Then $HN$ is an $\ast$expansive subgroup of $G$, $H\cap N$ is a normal $\ast$expansive subgroup of $H$ and $HN/N\simeq H/H\cap N$.
	\end{enumerate}
\end{prop}

%
%

Subnormal series, their refinements and equivalence of subnormal series for $\ast$expansive profinite groups are defined as for expansive profinite groups.

\begin{prop} \label{prop: astschreier refinement}
	Any two subnormal series of a $\ast$expansive profinite group have equivalent refinements. 
\end{prop}

The following proposition allows us to associate a $\ast$expansive profinite group $G^*$ to any expansive profinite group $G$. 

\begin{prop}
	Let $G$ be an expansive profinite group. There exists a $\ast$expansive profinite group $G^*$ together with a morphism $G^*\to G$ of topological $\s$-groups satisfying the following universal property: If $H$ is a $\ast$expansive profinite group and $H\to G$ is a morphism of topological $\s$-groups, then there exists a unique morphism $H\to G^*$ such that
	$$
	\xymatrix{
		H \ar[rr] \ar@{..>}[rd] & & G \\
		& G^* \ar[ru] &
	}
	$$
	commutes.
\end{prop}
\begin{proof}
	For $i\in\nn$ let $G_i$ be a copy of $G$ and consider the projective system $(G_i,\ \f_{i+1})_{i\in\nn}$, where the connection maps $\f_{i+1}\colon G_{i+1}\to G_i$ are all equal to $\s\colon G\to G$. Let $G^*$ be the projective limit of this projective system. Explicitly, we have
	$$G^*=\left\{(g_0,g_1,g_2,\ldots)\in G^\nn \big|\ \s(g_{i+1})=g_i \ \forall \ i\in\nn \right\}.$$
	Define $\s\colon G^*\to G^*$ by $\s(g_0,g_1.\ldots)=(\s(g_0),\s(g_1),\ldots)$. Then $\s\colon G^*\to G^*$ is continuous because the maps $G^*\to G,\ (g_0,g_1,\ldots)\mapsto \s(g_i)$ are continuous for every $i\in\nn$.
	If $g^*=(g_0,g_1,\ldots)$ lies in the kernel of $\s\colon G^*\to G^*$, then $\s(g_i)=1$ for all $i\in\nn$. But $\s(g_i)=g_{i-1}$ for $i\geq 1$. So $g^*=1$ and $\s$ is injective. On the other hand, if $g^*=(g_0,g_1,\ldots)\in G^*$, then $\s((g_1,g_2,\ldots))=g^*$, so $\s\colon G^*\to G^*$ is surjective. A bijective morphism of profinite groups is an automorphism because any surjective morphism of profinite groups is open (\cite[Remark~1.2.1~(f)]{FriedJarden:FieldArithmetic}). Thus $\s\colon G^*\to G^*$ is an automorphism. 
	
	The projection $\pi\colon G^*\to G,\ (g_0,g_1,\ldots)\mapsto g_0$ is a morphism of topological $\s$-groups. Let $U$ be an open subgroup of $G$ such that $\cap_{n\in\nn}\s^{-n}(U)=1$. Set $U^*=\pi^{-1}(U)$. We claim that $\cap_{n\in\Z}\s^n(U^*)=1$. Assume that $g^*=(g_0,g_1,\ldots)\in \cap_{n\in\Z}\s^n(U^*)$. Let $n,i\in\nn$. As $g^*\in\s^{i-n}(U^*)$, we see that $$\s^{n-i}(g^*)=\s^n(g_{i},g_{i+1},\ldots)=(\s^n(g_i),\s^n(g_{i+1}),\ldots)$$ lies in $U^*$, i.e., $\s^n(g_i)\in U$. So $g_i\in \s^{-n}(U)$ for all $n\in\nn$. Thus $g_i=1$ for all $i\in\nn$ and $g^*=1$ as desired. Therefore $G^*$ is a $\ast$expansive profinite group.
	
	Let $H$ be a $\ast$expansive profinite group and $\f\colon H\to G$ a morphism of profinite $\s$\=/groups.
	Define $\psi\colon H\to G^*$ by $\psi(h)=(\f(h),\f(\s^{-1}(h)),\f(\s^{-2}(h)),\ldots)$. Then $\psi$ is a morphism of topological $\s$-group such that 
	$$
	\xymatrix{
		H \ar^\f[rr] \ar@{..>}_-\psi[rd] & & G \\
		& G^* \ar_\pi[ru] &
	}
	$$
	commutes. Indeed, $\psi$ is the only such morphism, because any other morphism $\psi'\colon H\to G^*$ with this property satisfies $\pi(\s^{-i}(\psi'(h)))=\pi(\psi'(\s^{-i}(h)))=\f(\s^{-i}(h))$ for all $i\in\nn$.
\end{proof}

\begin{ex} \label{ex: ast for full}
	If $G=\G^\nn$ is the full one-sided group shift on a finite group $\G$, then $G^*=\G^\Z$ is the full two-sided group shift on $\G$ and $G^*\to G,\ (g_n)_{n\in\Z}\mapsto (g_n)_{n\in\nn}$ is the projection.
\end{ex}

The following example generalizes Example \ref{ex: ast for full}.
\begin{ex} \label{ex: astgamma}
	Let $\Gamma$ be a directed group graph on the finite group $\G$. Then $(G_\Gamma)^*=G_\Gamma^*$. (Recall that $G_\Gamma$ is defined after Definition \ref{def: group graph} and $G_\Gamma^*$ is defined before Definition \ref{defi: astexpansive group}.) 
\end{ex}
\begin{proof}
	We have a natural map $\pi\colon G_\Gamma^*\to G_\Gamma$ that associates to the vertices of a biinfinite directed path the vertices of the infinite directed subpath starting at the vertex in position zero. Let $H$ be a $\ast$expansive profinite group and $\f\colon H\to G_\Gamma$ a morphism of topological $\s$-groups. For every $h\in H$, the first vertex of the path corresponding to $\f(\s^{-1}(h))$ extends the path corresponding to $\f(h)$ one step further to the left, similarly for $\s^{-n}$ in place of $\s^{-1}$. So we see that there exists a unique $g=\psi(h)\in G_\Gamma^*$ such that $\f(\s^{-n}(h))=\pi(\s^{-n}(g))$ for all $n\in\N$.
\end{proof}

\begin{cor} \label{cor: comes from below}
	Every $\ast$expansive profinite group $H$ is of the form $H=G^*$ for some expansive profinite group $G$.
\end{cor}
\begin{proof}
	By \cite[Theorem 1, (i)]{Kitchens:ExpansiveDynamicsOnZeroDimensionalGroups} every $\ast$expansive profinite group $H$ is isomorphic to $G_\Gamma^*$ for some directed group graph $\Gamma$. By Example \ref{ex: astgamma} we can thus take $G=G_\Gamma$.
\end{proof}

\begin{ex} \label{ex: ast for sinfinitesimal}
	Let $G$ be a $\s$-infinitesimal expansive profinite group. Then $G^*=1$. Indeed, any morphism $\f\colon H\to G$ with $H$ a $\ast$expansive profinite group satisfies $\f(H)=1$.
\end{ex}
\begin{proof}
	By Lemma \ref{lemma: sinfinitesimal} there exists an $n\in\N$ such that $\s^n(g)=1$ for all $g\in G$. If $h\in H$ then we can write $h=\s^n(h')$ for some $h'\in H$ and then $\f(h)=\s^n(\f(h'))=1$.
\end{proof}

\begin{ex} \label{ex:finite}
	If $G$ is a finite (discrete) group with an endomorphism $\s\colon G\to G$, then $G^*=\cap_{n\in\N} \s^n(G)$. In particular, if $\s\colon G\to G$ is an automorphism, then $G^*=G$.
\end{ex}

%

\begin{lemma} \label{lemma: ast and subgroups}
	If $H$ is an expansive subgroup of an expansive profinite group $G$, then $H^*$ is a $\ast$expansive subgroup of $G^*$. Moreover, if $H$ is normal in $G$, then $H^*$ is normal in $G^*$ and $G^*/H^*\simeq (G/H)^*$.
\end{lemma}
\begin{proof}
	Clearly $H^*=\{(h_0,h_1,\ldots)\in H^\N|\ \s(h_{i+1})=h_i \ \forall \ i\in\N\}$ is a $\ast$expansive subgroup of $G^*=\{(g_0,g_1,\ldots)\in G^\N|\ \s(g_{i+1})=g_i \ \forall \ i\in\N\}$. If $H$ is normal in $G$, then $H^*$ is normal in $G^*$. The map $G^*/H^*\to (G/H)^*,\ \overline{(g_0,g_1,\ldots)}\mapsto (\overline{g_0},\overline{g_1},\ldots)$ is an isomorphism. 
\end{proof}

We will also need a $\ast$version of Lemma \ref{lemma: normal subgroup of shift on simple group}.

\begin{lemma} \label{lemma: normal subgroup of shift on simple group for inversive}
	Let $\G$ be a finite simple group and $G=\G^\Z$ the full two-sided group shift on $\G$. Let $N$ be a proper $\ast$expansive subgroup of $G$.
	\begin{enumerate}
		\item If $\G$ is non-commutative, then $N$ is trivial.
		\item If $\G$ is commutative, then $N$ is finite and $G/N$ is isomorphic to $G$.
	\end{enumerate}
	So, in either case, $N$ is finite and $G/N$ is isomorphic to $G$.
\end{lemma}
\begin{proof}
	For $i\in\nn$ let $N[i]$ denote the image of $N$ under the projection $\G^\Z\to \G^{i+1},\ (g_n)_{n\in\N}\mapsto (g_0,\ldots,g_i)$. In other words, $N[i]$ is the subgroup of $\G^{i+1}$ consisting of all blocks of length $i+1$ that occur in elements of $N$. As $N$ is normal in $G$, we see that $N[i]$ is normal in $\G^{i+1}$. Moreover, since $N$ is a proper subgroup of $G$, there must exist an $i\in\N$ such that $N[i]$ is a proper subgroup of $\G^{i+1}$.
	
	We first treat the case that $\G$ is non-commutative. Then, as explained in the proof of Lemma~\ref{lemma: normal subgroup of shift on simple group}, the group $N[i]$ is an $i+1$-fold product, where each factor is either $1$ or $\G$. In particular, one of the factors, say, the $j$-th factor, has to be $1$. This means that every block of length $i+1$ that occurs in an element of $N$ has a $1$ in its $j$-th position. But every entry of an element of $N$ is the $j$-th entry of some block of length $i+1$, so $N=1$.
	
	We next treat the commutative case. So $\G$ is cyclic of prime order. For every $i\geq 1$ the kernel of the projection $N[i]\to N[i-1],\ (g_0,\ldots,g_i)\mapsto (g_0,\ldots,g_{i-1})$ is of the form $\{1\}^i\times \mathcal{N}_i$ for some subgroup $\mathcal{N}_i$ of $\G$, i.e., $\mathcal{N}_i=1$ ore $\mathcal{N}_i=\G$. Set $\mathcal{N}_0=N[0]$. There exists an $n\in\N$ such that $\mathcal{N}_0,\ldots,\mathcal{N}_{n-1}$ are all equal to $\G$ and $\mathcal{N}_{n},\mathcal{N}_{n+1},\ldots$ are all trivial. So $N[i]=|\G|^n$ for all $i\geq n-1$. It follows that $|N|=|\G|^n$, in particular, $N$ is finite. The group $\G^{n+1}/N[n]$ is isomorphic to $\G$ because it has the same order. The map $G\to (\G^{n+1}/N[n])^\Z,\ (g_m)_{m\in\Z}\mapsto (\overline{(g_m,g_{m+1},\ldots,g_{m+n})})_{m\in\Z}$ has kernel $N$ and thus induces an isomorphism $G/N\simeq \G^\Z$.
\end{proof}

We are now prepared to establish the $\ast$version of Theorem \ref{theo: main}. 

\begin{theo} \label{theo: expansive auto}
	Let $G$ be a $\ast$expansive profinite group. Then there exists a subnormal series
	$$G\supseteq G_1\supseteq G_2\supseteq\ldots\supseteq G_n=1$$
	of $\ast$expansive subgroups $G_i$ of $G$ such that  $G/G_1$ is a finite group and
$G_i/G_{i+1}$ is isomorphic to a full two-sided group shift on a finite simple group $\G_i$ for $i=1,\ldots,n-1$. 
	Moreover, the length $n$ of such a series and the isomorphism classes of the finite simple groups $\G_i$ are uniquely determined by $G$. 
\end{theo}
\begin{proof}
	We know from Corollary \ref{cor: comes from below} that $G$ is of the form $G=H^*$ for some expansive profinite group $H$. Let 
	$$H\supseteq H_1\supseteq \ldots\supseteq H_{n}$$ be a subnormal series for $H$ as in Theorem \ref{theo: main}. For $i=1,\ldots,n $ set $G_i=H_i^*$. By Lemma \ref{lemma: ast and subgroups} we have a subnormal series
	\begin{equation} \label{eq: ast series}
		G\supseteq G_1\supseteq\ldots\supseteq G_n
	\end{equation}
	of $\ast$expansive subgroup for $G=H^*$. As $H/H_1$ is finite and $\s\colon H/H_1\to H/H_1$ is bijective (Lemma \ref{lemma: sconnected component} (ii)) we see, using Lemma \ref{lemma: ast and subgroups} and Example \ref{ex:finite}, that $G/G_1=H^*/H_1^*=(H/H_1)^*=H/H_1$ is finite. It follows from Example \ref{ex: ast for full} that $G_i/G_{i+1}=H_i^*/H_{i+1}^*=(H_i/H_{i+1})^*$ is isomorphic to a full two-sided group shift on a finite simple group $\G_i$ for $i=1,\ldots,n-1$. Moreover, by Example \ref{ex: ast for sinfinitesimal} we have $G_n=H_n^*=1$. So the subnormal series (\ref{eq: ast series}) has all the required properties. 
	
	The proof of the uniqueness claim is similar to the proof in Theorem \ref{theo: main}. Assume we have another subnormal series
	$$G\supseteq H_1\supseteq\ldots\supseteq H_m=1$$
	as in the theorem and let $\H_1.\ldots,\H_{m-1}$ denote the corresponding finite simple groups. According to Proposition \ref{prop: astschreier refinement}, we can find equivalent refinements. Using Lemma \ref{lemma: normal subgroup of shift on simple group for inversive} we see that the number of infinite factor groups in a refinement equals $n-1$ respectively $m-1$. Because the refinements are equivalent, we must have $m=n$. Using Lemma \ref{lemma: normal subgroup of shift on simple group for inversive} again, we see that there exists a permutation $\pi$ such that $\G_i^\Z$ is isomorphic to $\H_{\pi(i)}^\Z$. But then $\G_i\simeq \H_i$ (\cite[Proposition~7]{Kitchens:ExpansiveDynamicsOnZeroDimensionalGroups}).
\end{proof}

%
%
%

\section{Babbitt's decomposition}

Babbitt's decomposition theorem is an important classical theorem in difference algebra that elucidates the structure of algebraic difference field extensions. See \cite[Chapter~7, Theorem~7]{Cohn:difference}, \cite[Theorem 5.4.13]{Levin} or \cite[Theorem 2.3]{Babbitt:FinitelyGeneratedPathologicalExtensionsOfDifferenceFields} for the original reference.

In this section we explain how our main result (Theorem \ref{theo: main}) implies Babbitt's decomposition theorem and indeed yields additional information concerning the uniqueness of the decomposition.

To state Babbitt's decomposition theorem we need to recall some basic notation from difference algebra. See \cite{Cohn:difference} or \cite{Levin}.
A \emph{difference field}, or \emph{$\s$-field} for short, is a field $K$ equipped with an endomorphism $\s\colon K\to K$. An extension $L/K$ of difference fields is an extension of fields such that $\s\colon L\to L$ extends $\s\colon K\to K$. 
An \emph{intermediate $\s$-field} of a $\s$-field extension $L/K$ is a subfield $M$ of $L$ containing $K$ such that $\s(M)\subseteq M$. If $L/K$ is an extension of $\s$-fields and $A$ a subset of $L$, then $K\langle A\rangle\subseteq L$ denotes the smallest intermediate $\s$-field of $L/K$ that contains $A$. Note that $K\langle A\rangle=K(A,\s(A),\s^2(A)\ldots)$, the field extension of $K$ generated by $A,\s(A),\ldots$. If $L=K\langle A\rangle$ for a finite set $A$, then $L/K$ is called \emph{finitely $\s$-generated}.

A $\s$-field extension $L/K$ is Galois if the underlying field extension is Galois, i.e., normal and separable. (So the field extension is algebraic but not necessarily finite.) The following lemma explains the connection between extensions of difference fields and expansive profinite groups. See \cite[Section 8.1]{Levin} for related results in a slightly different context. (In \cite[Section~8.1]{Levin} it is always assumed that $\s\colon K\to K$ is an automorphism.)

\begin{lemma} \label{lemma: expansive structure on Galois group}
	Let $L/K$ be a Galois extension of $\s$-fields and let $G=G(L/K)$ be the Galois group (of the underlying field extension).
	\begin{enumerate}
		\item For every $g\in G$ there exists a unique $g^\sigma\in G$ such that $\s g^\sigma=g\sigma$ as maps from $L$ to $L$.
		\item The map $\s\colon G\to G,\ g\mapsto g^\s$ is a continuous group homomorphism.
		\item The extension $L/K$ is finitely $\s$-generated if and only if $G$ is an expansive profinite group.
	\end{enumerate}
\end{lemma} 
\begin{proof}
	In \cite[Lemma 1.23]{TomasivWibmer:StronglyEtaleDifferenceAlgebras} it is shown that for any two extensions $\s_1,\s_2\colon L\to L$ of $\s\colon K\to K$, there exists a unique element $\tau\in G$ such that $\s_2=\s_1\tau$. Applying this with $\sigma_1=\sigma\colon L\to L$ and $\s_2=g\s\colon L\to L$ yields (i).
	
	A basis for the topology of $G$ is given by the open subgroups $$U_A=\{g\in G|\ g(a)=a \ \forall\ a\in A\},$$ where $A$ is a finite subset of $L$. For $a\in L$ and $g\in G$ we have $\s(g)(a)=a$ if and only if $\s(\s(g)(a))=\s(a)$ because $\s\colon L\to L$ is injective. But $\s(\s(g)(a))=g(\s(a))$ by definition of $\s\colon G\to G$. It follows that $\s^{-1}(U_A)=U_{\s(A)}$. Therefore $\s\colon G\to G$ is continuous. A straight forward calculations shows that $\s\colon G\to G$ is a group homomorphism.
	
	To establish (iii), assume first that $A$ is a finite subset of $L$ such that $L=K\langle A\rangle$. We claim that $\cap_{n\in\nn}\s^{-n}(U_A)=1$. Indeed, if $g\in\cap_{n\in\nn}\s^{-n}(U_A)$, then $g\in \s^{-n}(U_A)=U_{\s^n(A)}$ for all $n\in\nn$. So $g$ fixes all elements in $A,\s(A),\s^2(A),\ldots$. Since these generate $L$ as a field extension of $K$, we see that $g$ fixes all elements of $L$. Thus $g=1$ and $G$ is an expansive profinite group.
	
	Conversely, assume that $U$ is an open subgroup of $G$ such that $\cap_{n\in\nn}\s^{-n}(U)=1$. Then $L^U$ is a finite field extension of $K$ and therefore of the form $L^U=K(A)$ for a finite subset $A$ of $L$. In other words, $U=U_A$. As $L/K(A,\s(A),\ldots)$ has Galois group $\cap_{n\in\nn}\s^{-n}(U)=1$, we must have $L=K\langle A\rangle$.
\end{proof}

To state Babbitt's decomposition theorem we need some more notation from difference algebra. An extension $L/K$ of $\s$-fields is \emph{$\s$-separable} if $\s\colon L\otimes_K K'\to L\otimes_K K',\ a\otimes b\mapsto \s(a)\otimes\s(b)$ is injective for any $\s$-field extension $K'/K$. For other equivalent characterizations of this notion see \cite[Section 1.1]{TomasivWibmer:StronglyEtaleDifferenceAlgebras}.

Let $L/K$ be a finitely $\s$-generated Galois extension of $\s$-fields and let $\pi_0^\s(L/K)$ denote the union of all intermediate $\s$-fields $M$ of $L/K$ such that $M$ is a finite field extension of $K$ and $M/K$ is $\s$-separable. Then $\pi_0^\s(L/K)/K$ is a finite field extension (\cite[Remarkark~1.27]{TomasivWibmer:StronglyEtaleDifferenceAlgebras}) and Galois (\cite[Corollary 1.35]{TomasivWibmer:StronglyEtaleDifferenceAlgebras}). It follows that $\pi_0^\s(L/K)$ is the largest intermediate $\s$-field of $L/K$ with the property that $\pi_0^\s(L/K)/K$ is finite Galois and $\s$\=/separable.

A Galois extension $L/K$ of $\s$-fields is \emph{benign} if there exists an intermediate field $K\subseteq M\subseteq L$ such that 
\begin{itemize}
	\item $M/K$ is a finite Galois extension with $L=K\langle M\rangle$,
	\item the degree of $K(\s^n(M))$ over $K$ equals the degree of $M$ over $K$ for all $n\in\nn$ and
	\item the fields $(K(\s^n(M)))_{n\in\nn}$ are linearly disjoint over $K$.
\end{itemize}
An extension $L/K$ of $\s$-fields is \emph{$\s$-radicial} if for every $a\in L$ there exists an $n\in\nn$ such that $\s^n(a)\in K$. 

Now we are prepared to state Babbitt's decomposition theorem. The version we state here is from \cite[Theorem 2.9]{TomasivWibmer:StronglyEtaleDifferenceAlgebras} and in contrast to the references given at the beginning of this section does not require $\s\colon K\to K$ to be an automorphism.

\begin{theo}[Babbitt's decomposition theorem] \label{theo: Babbitts decomposition}
	Let $L/K$ be a finitely $\s$-generated Galois extension of $\s$-fields. Then there exists a chain
	$$K\subseteq L_1\subseteq L_2\subseteq\ldots\subseteq L_n\subseteq L$$
	of intermediate $\s$-fields
	such that $L_1=\pi_0^\s(L/K)$, $L_{i+1}/L_{i}$ is benign for $i=1,\ldots,n-1$ and $L/L_n$ is $\s$-radicial.
\end{theo}

To deduce Theorem \ref{theo: Babbitts decomposition} from Theorem \ref{theo: main} we need to know how properties of expansive profinite groups correspond to properties of $\s$-field extensions. This is explained in the following lemma.

\begin{lemma} \label{lemma: properties correspond}
	Let $L/K$ be a finitely $\s$-generated Galois extension of $\s$-fields and let $G$ be the Galois group of $L/K$, considered as an expansive profinite group as in Lemma \ref{lemma: expansive structure on Galois group}.
	\begin{enumerate}
		\item $L/K$ is $\s$-separable if and only if $\s\colon G\to G$ is surjective.
		\item $L/K$ is finite and $\s$-separable if and only if $G$ is finite and $\s\colon G\to G$ is bijective.
		\item $L/K$ is benign if and only if $G$ is isomorphic to a full one-sided group shift.
		\item $L/K$ is $\s$-radicial if and only if $G$ is $\s$-infinitesimal.
	\end{enumerate}
\end{lemma}
\begin{proof}
	We begin with (i). By \cite[Proposition 1.2]{TomasivWibmer:StronglyEtaleDifferenceAlgebras} the $\s$-field extension $L/K$ is $\s$\=/separable if and only if whenever $f_1,\ldots,f_n\in L$ are $K$-linearly independent then also $\s(f_1),\ldots,\s(f_n)$ are $K$-linearly independent.
	
	Assume that $L/K$ is $\s$-separable. To show that $\s\colon G\to G$ is surjective, it suffices to show that $G\xrightarrow{\s} G\to G/U$ is surjective for every normal open subgroup $U$ of $G$. With notation as in the proof of Lemma \ref{lemma: expansive structure on Galois group}, we have $U=U_A$ for some finite subset $A$ of $L$. So $G/U$ can be identified with the Galois group of $K(A)/K$. As $\s^{-1}(U_A)=U_{\s(A)}$, we see that $G/\s^{-1}(U)$ can be identified with the Galois group of $K(\s(A))/K$. Because $L/K$ is $\s$-separable, we see that $K(A)$ and $K(\s(A))$ have the same degree over $K$. Since the map $G/\s^{-1}(U)\to G/U$ induced by $\s\colon G\to G$ is injective it must then be surjective. Therefore $\s\colon G\to G$ is surjective.
	
	Conversely, assume that $\s\colon G\to G$ is surjective. By the primitive element theorem it suffices to show that for any element $a\in L$ such that $K(a)/K$ is Galois, the fields $K(a)$ and $K(\s(a))$ have the same degree over $K$. Let $n$ be the degree of $K(a)/K$. Then $a$ has $n$ conjugates $g_1(a),\ldots,g_n(a)\in L$. We have to show that $\s(g_1(a)),\ldots,\s(g_n(a))\in L$ are also conjugate over $K$. Because $\s\colon G\to G$ is surjective we may write $g_i=\s(h_i)$ for some $h_i\in G$. Then $\s(g_i(a))=\s(\s(h_i)(a))=h_i(\s(a))$ and therefore these elements are conjugate over $K$.
	
	Point (ii) follows from (i), because for $G$ finite, $\s\colon G\to G$ is surjective if and only if it is bijective.

We next prove (iii). Assume that $L/K$ is benign and let $K\subseteq M\subseteq L$ be a finite Galois extension of $K$ such that $L=K\langle M\rangle$, the degree of $K(\s^n(M))$ over $K$ equals the degree of $M$ over $K$ for all $n\in\nn$ and the fields $(K(\s^n(M)))_{n\in\nn}$ are linearly disjoint over $K$. Let $A\subseteq M$ be finite such that $M=K(A)$. Set $U=U_A$. Then $G/U$ can be identified with the Galois group of $M/K$. More generally, as $\s^{-n}(U_A)=U_{\s^n(A)}$, we see that $G/\s^{-n}(U)$ can be identified with the Galois group of $K(\s^n(A))=K(\s^n(M))$ over $K$. As $L=K\langle M\rangle$ and the fields $(K(\s^n(M)))_{n\in\nn}$ are linearly disjoint over $K$, the canonical map $G\to \prod_{n\in\nn}G/\s^{-n}(U)$ is an isomorphism of profinite groups. As $M$ and $K(\s^n(M))$ have the same degree over $K$ for all $n\in\nn$, the map $\s_{n+1}\colon G/\s^{-(n+1)}(U)\to G/\s^{-n}(U)$ induced by $\s\colon G\to G$ is an isomorphism. If we define $$\s\colon \prod_{n\in\nn}G/\s^{-n}(U)\to \prod_{n\in\nn}G/\s^{-n}(U),\ (g_n)_{n\in\nn}\mapsto (\s_{n+1}(g_{n+1}))_{n\in\nn},$$ then  $G\to \prod_{n\in\nn}G/\s^{-n}(U)$ becomes an isomorphism of expansive profinite groups. But $\prod_{n\in\nn}G/\s^{-n}(U)$ is isomorphic to the full one-sided group shift on $G/U$.

Conversely, assume that $G$ is isomorphic to the full one-sided group shift on a finite group $\G$. Let $U$ be the open normal subgroup of $G$ corresponding to $1\times\G\times\G\times\ldots\leq \G^\nn$ and define $M$ as $L^U$, so $M/K$ has Galois group $G/U=\G$. 
Then $M$ has all the required properties.

Finally, we prove (iv). Assume that $L/K$ is $\s$-radicial. Let $A$ be finite subset of $L$ such that $L=K\langle A\rangle$. As $L/K$ is $\s$-radicial there exists an $n\in\nn$ such that $\s^n(A)\in K$. But then in fact $\s^n(L)\subseteq K$. For $a\in L$ and $g\in G$ we have $\s(\s(g)(a))=g(\s(a))$ and so inductively $\s^n(\s^n(g)(a))=g(\s^n(a)=\s^n(a)$ because $\s^n(a)\in K$. The injectivity of $\s\colon L\to L$ implies $\s^n(g)(a)=a$ for all $a\in L$, i.e., $\s^n(g)=1
$.

Conversely, assume that $G$ is $\s$-infinitesimal. By Lemma \ref{lemma: sinfinitesimal} there exists an integer $n\in\nn$ such that $\s^n(g)=1$ for all $g\in G$. Then $g(\s^n(a))=\s^n(\s^n(g)(a))=\s^n(a)$ for all $g\in G$ and $a\in L$. Thus $\s^n(a)\in K$.
\end{proof}

The following lemma is a ``one-sided'' version of \cite[Theorem 8.1.1]{Levin}.
 \begin{lemma} \label{lemma: Galois correspondence}
 	Let $L/K$ be a finitely $\s$-generated Galois extension of $\s$-fields and consider the Galois group $G=G(L/K)$ of $L/K$ as an expansive profinite group as in Lemma~\ref{lemma: expansive structure on Galois group}. 
 	\begin{enumerate}
 		\item The maps $M\mapsto G(L/M)$ and $H\mapsto L^H$ define a bijection between the intermediate $\s$-fields of $L/K$ and the expansive subgroups of $G$.
 		\item  If $M$ and $H$ correspond to each other in {\rm (i)}, then $M/K$ is Galois if and only if $H$ is normal in $G$ and in that case, $G(M/K)$ and $G/H$ are isomorphic as expansive profinite groups. 
 		\item The expansive subgroup $G^{\sc}$ of $G$ corresponds to the intermediate $\s$-field $\pi_0^\s(L/K)$.
 	\end{enumerate}
 \end{lemma}
\begin{proof}
	By the Galois correspondence, the map $M\mapsto G(L/M)$ is a bijection between the set of all intermediate fields of $L/K$ and the set of closed subgroups of $G$ with inverse $H\mapsto L^H$. So, if $M$ corresponds to $H$, it suffices to show that $\s(M)\subseteq M$ if and only if $\s(H)\subseteq H$. First assume that $\s(M)\subseteq M$. Then, for $a\in M$ and $h\in H$ we have $\s(\s(h)(a))=h(\s(a))=\s(a)$ because $\s(a)\in M$. From the injectivity of $\s\colon L\to L$ it follows that $\s(h)(a)=a$ for all $a\in M$, i.e., $\s(h)\in H$. Conversely, assume that $\s(H)\subseteq H$ and $a\in M$. Then $\s(h)(a)=a$ and therefore $\s(a)=\s(\s(h)(a))=h(\s(a))$ for all $h\in H$. Thus $\s(a)\in M$.
	
	Part (ii) is clear from Galois theory. The only aspect that needs to be checked is that the restriction map $G\to G(M/K)$ commutes with $\s$, but this follows directly from the definition of the action of $\s$.
	
	We know from Lemma \ref{lemma: sconnected converse} that $G^\sc$ is the smallest normal expansive subgroup of $G$ such that $G/G^\sc$ is finite and $\s\colon G/G^\sc\to G/G^\sc$ is bijective. Thus, by Lemmas \ref{lemma: properties correspond} and \ref{lemma: Galois correspondence}, $L^{G^\sc}$ is the largest Galois extension of $K$ such that $L^{G^\sc}/K$ is finite and $\s$-separable. But this is exactly $\pi_0^\s(L/K)$.
\end{proof}

With these preparations at hand, it is now a straight forward matter to deduce Babbitt's decomposition theorem from Theorem \ref{theo: main}. 

\begin{theo} \label{theo: Babbitt new}
	Let $L/K$ be a finitely $\s$-generated Galois extension of $\s$-fields. Then there exists a chain
	$$K\subseteq L_1\subseteq L_2\subseteq\ldots\subseteq L_n\subseteq L$$
	of intermediate $\s$-fields
	such that $L_1=\pi_0^\s(L/K)$, $L_{i+1}/L_{i}$ is benign with Galois group isomorphic to a full one-sided group shift on a finite simple group $\G_i$ for $i=1,\ldots,n-1$ and $L/L_n$ is $\s$\=/radicial.
	Moreover, the length $n$ of such a chain and the isomorphism classes of the finite simple groups $\G_i$ are uniquely determined by $L/K$.
\end{theo}
\begin{proof}
	Let $G=G(L/K)$ denote the Galois group of $L/K$, considered as an expansive profinite group as in Lemma \ref{lemma: expansive structure on Galois group}. Let 
	$$ G\supseteq G_1\supseteq\ldots\supseteq G_n$$
	be a subnormal series as in Theorem \ref{theo: main}. For $i=1,\ldots,n$ set $L_i=L^{G_i}$.
	Then $$K\subseteq L_1\subseteq\ldots\subseteq L_n\subseteq L$$ is an ascending chain of intermediate $\s$-field that has the required properties by Lemmas~\ref{lemma: Galois correspondence} and \ref{lemma: properties correspond}.

	If we have another chain
	 $$K\subseteq L'_1\subseteq\ldots\subseteq L'_{n'}\subseteq L$$ as in the theorem, then setting $G_i'=G(L/L'_i)$ for $i=1,\ldots,n'$ yields a subnormal series 
	 $$ G\supseteq G'_1\supseteq\ldots\supseteq G'_{n'}$$
	 as in Theorem \ref{theo: main}. The uniqueness part of Theorem \ref{theo: main} thus implies the claimed uniqueness statement.	
\end{proof}

\bibliographystyle{alpha}
\bibliography{bibdata}

\def\cprime{$'$}
\begin{thebibliography}{Coh65}

\bibitem[Bab62]{Babbitt:FinitelyGeneratedPathologicalExtensionsOfDifferenceFields}
Albert~E. Babbitt, Jr.
\newblock Finitely generated pathological extensions of difference fields.
\newblock {\em Trans. Amer. Math. Soc.}, 102:63--81, 1962.

\bibitem[BS08]{BoyleSchrauder:ZdgroupshiftsAndBernoulliFactors}
Mike Boyle and Michael Schraudner.
\newblock {$\Bbb Z^d$} group shifts and {B}ernoulli factors.
\newblock {\em Ergodic Theory Dynam. Systems}, 28(2):367--387, 2008.

\bibitem[Coh65]{Cohn:difference}
Richard~M. Cohn.
\newblock {\em Difference algebra}.
\newblock Interscience Publishers John Wiley \& Sons, New York-London-Sydeny,
  1965.

\bibitem[Fag96]{Fagnani:SomeResultsOnTheClassificationOfExpansiveAutomorphismsOfCompactAbelianGroups}
Fabio Fagnani.
\newblock Some results on the classification of expansive automorphisms of
  compact abelian groups.
\newblock {\em Ergodic Theory Dynam. Systems}, 16(1):45--50, 1996.

\bibitem[FJ08]{FriedJarden:FieldArithmetic}
Michael~D. Fried and Moshe Jarden.
\newblock {\em Field arithmetic}, volume~11 of {\em Results in Mathematics and
  Related Areas. 3rd Series. A Series of Modern Surveys in Mathematics}.
\newblock Springer-Verlag, Berlin, third edition, 2008.
\newblock Revised by Jarden.

\bibitem[GR17]{GloecknerRaja:ExpansiveAutomorphismsOfTotallyDisconnectedLocallyCompactGroups}
Helge Gl\"{o}ckner and C.~R.~E. Raja.
\newblock Expansive automorphisms of totally disconnected, locally compact
  groups.
\newblock {\em J. Group Theory}, 20(3):589--619, 2017.

\bibitem[Kit87]{Kitchens:ExpansiveDynamicsOnZeroDimensionalGroups}
Bruce~P. Kitchens.
\newblock Expansive dynamics on zero-dimensional groups.
\newblock {\em Ergodic Theory Dynam. Systems}, 7(2):249--261, 1987.

\bibitem[Kit98]{Kitchens:SymbolicDynamics}
Bruce~P. Kitchens.
\newblock {\em Symbolic dynamics}.
\newblock Universitext. Springer-Verlag, Berlin, 1998.
\newblock One-sided, two-sided and countable state Markov shifts.

\bibitem[KS89]{KitchensSchmidt:AutomorphismsOfCompactGroups}
Bruce Kitchens and Klaus Schmidt.
\newblock Automorphisms of compact groups.
\newblock {\em Ergodic Theory Dynam. Systems}, 9(4):691--735, 1989.

\bibitem[Lev08]{Levin}
Alexander Levin.
\newblock {\em Difference algebra}, volume~8 of {\em Algebra and Applications}.
\newblock Springer, New York, 2008.

\bibitem[LM95]{LindMarcus:IntroductionToSymbolicDynamisAndCoding}
Douglas Lind and Brian Marcus.
\newblock {\em An introduction to symbolic dynamics and coding}.
\newblock Cambridge University Press, Cambridge, 1995.

\bibitem[Rot95]{Rotman:AnIntroductionToTheTheoryOfGroups}
Joseph~J. Rotman.
\newblock {\em An introduction to the theory of groups}, volume 148 of {\em
  Graduate Texts in Mathematics}.
\newblock Springer-Verlag, New York, fourth edition, 1995.

\bibitem[RZ10]{RibesZaleskii:ProfiniteGroups}
Luis Ribes and Pavel Zalesskii.
\newblock {\em Profinite groups}, volume~40 of {\em Ergebnisse der Mathematik
  und ihrer Grenzgebiete. 3. Folge. A Series of Modern Surveys in Mathematics
  [Results in Mathematics and Related Areas. 3rd Series. A Series of Modern
  Surveys in Mathematics]}.
\newblock Springer-Verlag, Berlin, second edition, 2010.

\bibitem[Sch95]{Schmidt:DynamicalSystemsOfAlgebraicOrigin}
Klaus Schmidt.
\newblock {\em Dynamical systems of algebraic origin}, volume 128 of {\em
  Progress in Mathematics}.
\newblock Birkh\"{a}user Verlag, Basel, 1995.

\bibitem[Sha20]{Shah:ExpansiveAutomorphismsOnLocallyCompactGroups}
Riddh Shah.
\newblock Expansive automorphisms on locally compact groups.
\newblock {\em New York J. Math.}, 26:285--302, 2020.

\bibitem[Sob07]{Sobottka:TopologicalQuasiGroupShifts}
Marcelo Sobottka.
\newblock Topological quasi-group shifts.
\newblock {\em Discrete Contin. Dyn. Syst.}, 17(1):77--93, 2007.

\bibitem[TW18]{TomasivWibmer:StronglyEtaleDifferenceAlgebras}
Ivan Toma\v{s}i\'{c} and Michael Wibmer.
\newblock Strongly \'{e}tale difference algebras and {B}abbitt's decomposition.
\newblock {\em J. Algebra}, 504:10--38, 2018.

\end{thebibliography}

\end{document}